\documentclass[11pt]{amsart}

\usepackage{latexsym}
\usepackage{amssymb}
\usepackage{mathrsfs}
\usepackage{amsmath}
\usepackage{fancybox,color}
\usepackage{enumerate}

 \def\1{\raisebox{2pt}{\rm{$\chi$}}}

\newtheorem{theorem}{Theorem}[section]
\newtheorem{corollary}[theorem]{Corollary}
\newtheorem{lemma}[theorem]{Lemma}
\newtheorem{proposition}[theorem]{Proposition}
\newtheorem{definition}[theorem]{Definition}
\newtheorem{remark}[theorem]{Remark}
\newtheorem{example}[theorem]{Example}

\newcommand{\R}{{\mathbb R}}
\newcommand{\RR}{{\mathbb R}}

 \newcommand{\eps}{{\varepsilon}}
 \def\1{\raisebox{2pt}{\rm{$\chi$}}}
 
\newcommand{\Lip}{\operatorname{Lip}}
\newcommand{\abs}[1]{\left|#1\right|}
\newcommand{\norm}[1]{\left|\left|#1\right|\right|}
\newcommand{\Rn}{\mathbb{R}^n}

\def\vint_#1{\mathchoice%
          {\mathop{\kern 0.2em\vrule width 0.6em height 0.69678ex depth -0.58065ex
                  \kern -0.8em \intop}\nolimits_{\kern -0.4em#1}}%
          {\mathop{\kern 0.1em\vrule width 0.5em height 0.69678ex depth -0.60387ex
                  \kern -0.6em \intop}\nolimits_{#1}}%
          {\mathop{\kern 0.1em\vrule width 0.5em height 0.69678ex
              depth -0.60387ex
                  \kern -0.6em \intop}\nolimits_{#1}}%
          {\mathop{\kern 0.1em\vrule width 0.5em height 0.69678ex depth -0.60387ex
                  \kern -0.6em \intop}\nolimits_{#1}}}
\def\vintslides_#1{\mathchoice%
          {\mathop{\kern 0.1em\vrule width 0.5em height 0.697ex depth -0.581ex
                  \kern -0.6em \intop}\nolimits_{\kern -0.4em#1}}%
          {\mathop{\kern 0.1em\vrule width 0.3em height 0.697ex depth -0.604ex
                  \kern -0.4em \intop}\nolimits_{#1}}%
          {\mathop{\kern 0.1em\vrule width 0.3em height 0.697ex depth -0.604ex
                  \kern -0.4em \intop}\nolimits_{#1}}%
          {\mathop{\kern 0.1em\vrule width 0.3em height 0.697ex depth -0.604ex
                  \kern -0.4em \intop}\nolimits_{#1}}}

\newcommand{\aveint}[2]{\mathchoice%
          {\mathop{\kern 0.2em\vrule width 0.6em height 0.69678ex depth -0.58065ex
                  \kern -0.8em \intop}\nolimits_{\kern -0.45em#1}^{#2}}%
          {\mathop{\kern 0.1em\vrule width 0.5em height 0.69678ex depth -0.60387ex
                  \kern -0.6em \intop}\nolimits_{#1}^{#2}}%
          {\mathop{\kern 0.1em\vrule width 0.5em height 0.69678ex depth -0.60387ex
                  \kern -0.6em \intop}\nolimits_{#1}^{#2}}%
          {\mathop{\kern 0.1em\vrule width 0.5em height 0.69678ex depth -0.60387ex
                  \kern -0.6em \intop}\nolimits_{#1}^{#2}}}

\newcommand{\ud}{\, d}
\newcommand{\half}{{\frac{1}{2}}}

\newcommand{\ol}{\overline}
\newcommand{\Om}{\Omega}
\newcommand{\I}{\textrm{I}}
\newcommand{\II}{\textrm{II}}
\newcommand{\dist}{\operatorname{dist}}
\newcommand{\diver}{\operatorname{div}}
\newcommand{\inter}{\operatorname{int}}

\newcommand{\supp}{\operatorname{supp}}
\newcommand{\essliminf}{\operatornamewithlimits{ess\,liminf}}
\newcommand{\esslimsup}{\operatornamewithlimits{ess\,limsup}}
\newcommand{\lip}{{\rm Lip}}

\newcommand{\trm}{\textrm}
\newcommand{\vp}{\varphi}

\begin{document}

\title[Gradient constraints and the infinity
Laplacian]{\bf Discontinuous gradient constraints and the infinity
Laplacian}


\author[P. Juutinen, M. Parviainen and J. D. Rossi]
{Petri Juutinen, Mikko Parviainen and Julio D. Rossi}
\address{Petri Juutinen
\hfill\break\indent
Department of Mathematics and Statistics
\hfill\break\indent
University of Jyv\"askyl\"a
\hfill\break\indent P.O. Box 35, 40014 University of Jyv\"askyl\"a,
\hfill\break\indent Finland
\hfill\break\indent
{\tt petri.juutinen@jyu.fi}}

\address{Mikko Parviainen
\hfill\break\indent
Department of Mathematics and Statistics
\hfill\break\indent
University of Jyv\"askyl\"a
\hfill\break\indent P.O. Box 35, 40014 University of Jyv\"askyl\"a
\hfill\break\indent Finland
\hfill\break\indent
{\tt mikko.j.parviainen@jyu.fi}}

\address{Julio D. Rossi
\hfill\break\indent Departamento de An\'{a}lisis Matem\'{a}tico,
\hfill\break\indent  Universidad de Alicante,
\hfill\break\indent Ap. correos 99, 03080, Alicante, \hfill\break\indent
Spain.
\hfill\break\indent
On leave from Dpto. de Matem\'{a}ticas, FCEyN
\hfill\break\indent Universidad de Buenos Aires, 1428
\hfill\break\indent  Buenos
Aires, Argentina.
\hfill\break\indent {\tt jrossi@dm.uba.ar} }

\date{\today}


\begin{abstract}
Motivated by tug-of-war games and asymptotic analysis of certain variational problems,
we consider the following gradient constraint problem: given a
bounded domain $\Om\subset\Rn$, a continuous function
$f\colon\partial\Om\to\R$ and a non-empty subset $D\subset\Om$,
find a solution to
\[
\begin{cases}
 \min\{\Delta_\infty u, \abs{Du}-\chi_D\}=0&\text{in $\Om$}\\
u=f&\text{on $\partial\Om$},
\end{cases}
\]
where $\Delta_\infty$ is the infinity Laplace operator.
We prove that this problem always has a solution that is unique
if $\overline D=\overline{\inter D}$. If this regularity condition
on $D$ fails, then solutions obtained from game theory and
$L^p$-approximation need not coincide.
\end{abstract}

\maketitle

\section{Introduction}

The infinity Laplacian, introduced by
Aronsson
\cite{A} in 1960's, is a second order quasilinear partial
differential operator formally defined as
\[
\Delta_\infty u(x)=D^2u(x)Du(x)\cdot Du(x)=\sum_{i,j=1}^n u_{ij}(x)u_i(x)u_j(x).
\]
It is the ``Laplacian of $L^\infty$-variational problems'': the
equation $\Delta_\infty u(x)=0$ is the Euler-Lagrange equation for
the variational problem of finding absolute minimizers for the
prototypical $L^\infty$-functional
\[
I(u)=\norm{Du}_{L^\infty(\Om)}
\]
with given boundary values, see e.g.\ \cite{J}. The infinity
Laplacian also arises from certain random turn games \cite{PSSW}, \cite{BEJ} and
mass transportation problems \cite{gmpr}, and it appears in several
applications, such as image reconstruction and enhancement
\cite{cms}, and the study of shape metamorphism \cite{cepb}.

In this paper, we are interested in the following gradient
constraint problem involving the infinity Laplacian: given a
bounded domain $\Om\subset\Rn$, a continuous function
$f\colon\partial\Om\to\R$ and a non-empty subset $D\subset\Om$,
find a viscosity solution to
\begin{equation}\label{bvp:grad_constraint.intro}
\begin{cases}
 \min\{\Delta_\infty u(x), \abs{Du(x)}-\chi_D(x)\}=0&\text{in $\Om$}\\
u(x)=f(x)&\text{on $\partial\Om$}.
\end{cases}
\end{equation}
Here $\chi_D\colon\Om\to\R$ denotes the characteristic function of
the set $D$, that is,
$$
\chi_D(x)= \left\{\begin{array}{ll} 1, &
\text{ if $x\in
D$,} \\
0, & \text{if $x\in\Om\setminus D$.}
\end{array} \right.
$$

The study of gradient constraint problems of the type
\begin{equation}\label{eq:gen_grad_constraint}
  \min\{\Delta_\infty u(x), \abs{Du(x)}-g(x) \}=0,
\end{equation}
where $g\ge 0$, was initiated by Jensen in his celebrated paper
\cite{J}. He used the solutions of the equation
$\min\{\Delta_\infty u, \abs{Du}-\eps\}=0$ and its pair
$\max\{\Delta_\infty u, \eps-\abs{Du}\}=0$ to approximate the
solutions of the infinity Laplace equation $\Delta_\infty u=0$.
In this way, he proved uniqueness for the infinity Laplace equation by first showing  that it
holds for the approximating equations. The same approach was used
in the anisotropic case by Lindqvist and Lukkari in \cite{ll}, and
a variant of \eqref{eq:gen_grad_constraint} appears in the so
called $\infty$-eigenvalue problem, see e.g \cite{jlm}.

In general, the uniqueness of solutions for
\eqref{eq:gen_grad_constraint} is fairly easy to show if $g$ is
continuous and everywhere positive, and is known to hold, owing to
Jensen's work, if $g\equiv 0$. However, the case $g\ge 0$ seems to
have been entirely open before this paper. The situation resembles
the one with the infinity Poisson equation $\Delta_\infty u=g$:
the uniqueness is known to hold if $g>0$ or $g\equiv
0$, and the case $g\ge 0$ is an outstanding open problem, see
\cite{PSSW}. It is one of our main results in this paper that the
uniqueness for \eqref{eq:gen_grad_constraint} holds in the special
case $g=\chi_D$ under the fairly mild regularity condition
$\overline D=\overline{\inter D}$ on the set $D$, see Theorem
\ref{thm:unique1} below. Moreover, the uniqueness in general fails
if this condition is not fulfilled.

Our interest in \eqref{bvp:grad_constraint.intro} arises only
partially from the desire to generalize Jensen's results. Another
reason for considering this problem is its connection to the
boundary value problems
\begin{equation}\label{bvp:p_Lap_intro}
\begin{cases}
 \Delta_p u=g&\text{in $\Om$},\\
u=f&\text{on $\partial\Om$},
\end{cases}
\end{equation}
where $\Delta_p u=\diver(\abs{Du}^{p-2}Du)$ is the $p$-Laplace
operator, $1<p<\infty$, and $g\ge 0$. It is not difficult to show
that, up to selecting a subsequence, solutions $u_p$ to
\eqref{bvp:p_Lap_intro} converge as $p\to\infty$ to a limit
function that must satisfy \eqref{bvp:grad_constraint.intro} with
$D=\{x\in\Om\colon g(x)>0\}$. However, different subsequences may,
a priori, yield different limit functions. This possibility has
been previously excluded in the cases $g\equiv 0$ and $g>0$, the
latter under the additional assumption that $f=0$, see e.g.\
\cite{BBM}, \cite{IL}. Our results imply that the limit function
is unique for \emph{any} continuous functions $g\ge 0$ and 
$f$, see Theorem~\ref{theo.f.g}. In particular, the limit function depends on $g$ only via the
set $D=\{x\in\Om\colon g(x)>0\}$.

Needless to say, for $g\leq 0$ our techniques can be applied as
well, and then we encounter the equation
$$
\max\{ \Delta_\infty u, \chi_D - |Du| \} =0.
$$
Since the results are identical, we omit this case.

Further motivation for considering
\eqref{bvp:grad_constraint.intro} comes from its connection to
game theory. Recently, Peres, Schramm, Sheffield and Wilson
\cite{PSSW} introduced a two player random turn game called
``tug-of-war``, and showed that, as the step size converges to
zero, the value functions of this game converge to the unique
viscosity solution of the infinity Laplace equation $\Delta_\infty
u=0$. We define and study a variant of the tug-of-war game in
which one of the players has the option to sell his/hers turn to
the other player with a fixed price (that depends on the step
size) when the game token is in the set $D$. It is then shown that
the value functions of this new game converge to a solution of
\eqref{bvp:grad_constraint.intro}. Thus, besides its own interest, the game provides an
alternative way to prove the existence of a solution to
\eqref{bvp:grad_constraint.intro}.

The boundary value problem \eqref{bvp:grad_constraint.intro} may
have multiple solutions if the set $D$ is irregular, that is,
$\overline D\ne \overline{\inter D}$. However, the limit of the
value functions of our game is always the smallest solution and
hence  unique. We give several examples in which the game
solution and the solution constructed by taking the limit as $p$ goes
to infinity in the $p$-Laplace problems $\Delta_p u=\chi_D$ are
not the same. The possibility of having different solutions is
also motivated by stability considerations. Somewhat similar
results but on a different problem were recently obtained by
Yu in
\cite{y}.

 Our main uniqueness result, that
\eqref{bvp:grad_constraint.intro} has exactly one solution if
$\overline D= \overline{\inter D}$, is proved in a slightly
unusual manner. Indeed, instead of proving directly a comparison
principle for sub- and supersolutions of
\eqref{bvp:grad_constraint.intro}, we identify the solution in a way that guarantees its
uniqueness, see Theorem \ref{thm:charac} below for details. The intuition for this
identification comes partially
from the game theoretic interpretation of our problem. On the
other hand, the uniqueness proof for
\eqref{bvp:grad_constraint.intro} gave us a hint on how to prove similar result for the value
functions of the game, and so these two sides complement each
other nicely.

Due to the fact that the solutions need not be smooth and that the
infinity Laplacian is not in divergence form, we use viscosity
solutions when dealing with
\eqref{bvp:grad_constraint.intro}. However, since $\chi_D(x)$ can be viewed either as a
function, defined at every point of $\Om$, or as an element of
$L^\infty(\Om)$, defined only almost everywhere, one can use
either the standard notion or the $L^\infty$-viscosity solutions.
The first one fits well with the game theoretic approach, whereas
$L^\infty$-viscosity solutions are quite natural from the point of
view of $p$-Laplace approximation. We have chosen to use mostly
the standard notion of (continuous) viscosity solutions, mainly
because this makes it easier to compare our results with
what has been proved earlier.
For completeness, we
have included a short section explaining $L^\infty$-viscosity solutions of
\eqref{bvp:grad_constraint.intro}.

\tableofcontents

\section{Preliminaries}

\subsection{Viscosity solutions for the gradient constraint problem}

To be on the safe side, we begin by recalling what is meant by viscosity solutions
of the boundary value problem
\begin{equation}
\label{bvp:grad_constraint}
\begin{cases}
 \min\{\Delta_\infty u, \abs{Du}-\chi_D\}=0&\text{in $\Om$}\\
u=f&\text{on $\partial\Om$},
\end{cases}
\end{equation}
where $$\Delta_\infty u(x)= \sum_{i,j=1}^n u_{ij}(x)u_i(x)u_j(x)$$ is the infinity
Laplace operator and for $D\subset\Om$,
\[
\begin{split}
\chi_D(x)=
\begin{cases}
1,&x\in D,\\
0,& x\in \Om \setminus D.
\end{cases}
\end{split}
\]

First, the boundary condition ''$u=f$ on $\partial\Om$`` is understood in the classical
sense, that is, $\lim\limits_{x\to z} u(x)=f(z)$ for all $z\in\partial\Om$.

Second, to define viscosity solutions for the equation
\begin{equation}
\label{eq:grad_constraint}
 \min\{\Delta_\infty u, \abs{Du}-\chi_D\}=0,
\end{equation}
one needs to use the semicontinuous envelopes of $\chi_D$. To this
end, we denote by $\inter D$ and $\overline D$, respectively, the
(topological) interior and closure of the set~$D$.

\begin{definition}\label{def:viscosol}
An upper semicontinuous function $u\colon\Om\to\R$ is a
\emph{viscosity subsolution} to \eqref{eq:grad_constraint} in $\Om$ if, whenever
$ x \in\Om$ and $\varphi\in C^2(\Om)$ are such that
$u-\varphi$ has a strict local maximum at $ x$, then
\begin{equation}
\label{subineqs}
\min\{\Delta_\infty \vp(x), \abs{D\vp(x)}-\chi_{\inter {D}}(x)\}\ge 0.
\end{equation}

A lower semicontinuous function $v\colon\Om\to\R$ is a
\emph{viscosity supersolution} to \eqref{eq:grad_constraint} in
$\Om$ if, whenever $ x\in\Om$ and $\phi\in C^2(\Om)$ are such
that $v-\phi$ has a strict local minimum at $ x$, then
\begin{equation}\label{superineqs.22}
\min\{\Delta_\infty \phi(x), \abs{D\phi(x)}-\chi_{\ol {D}}(x)\}\le 0.
\end{equation}

Finally, a continuous function $h\colon\Om\to\R$ is a \emph{viscosity
solution} to \eqref{eq:grad_constraint} in $\Om$ if it is both a viscosity
subsolution and a viscosity supersolution.
\end{definition}

Sometimes it is convenient to replace the condition
"$u-\varphi$ has a strict local maximum at $x$" by the requirement that
"$\vp$ touches $u$ at $x$ from above", and to write
\eqref{subineqs} in the form
$$
\Delta_\infty \varphi(x)\ge 0\quad\text{and}\quad  \abs{D\varphi(x)}-\chi_{\inter D}(x)\ge 0.
$$
Similarly, we sometimes replace "$v-\phi$ has a strict local
minimum at $x$" by "$\phi$ touches $v$ at $x$ from below", and
write \eqref{superineqs.22} in the form
$$
\Delta_\infty \phi(x)\le
0\quad\text{or}\quad  \abs{D\phi(x)}-\chi_{\overline D}(x)\le 0.
$$

The reader should notice that if $\inter D$ is empty, then a solution $u$ to the infinity
Laplace equation $\Delta_\infty u=0$ also satisfies
 $\min\{\Delta_\infty u, \abs{Du}-\chi_D\}=0$.

\subsection{Patching and Jensen's equation}
The main difficulty in proving the uniqueness of solutions for the infinity Laplace equation
$\Delta_\infty u(x)=0$ is the very severe degeneracy of the equation at the points
where the gradient $Du$ vanishes. To overcome this, several approximation methods have been
introduced. The first one, due to Jensen \cite{J}, was to use the equation
\begin{equation}
\label{eq:jensen}
 \min\{\Delta_\infty u(x), \abs{Du(x)}-\eps\}=0,
\end{equation}
whose solutions are subsolutions of $\Delta_\infty u=0$ and have (at least formally)
a non-vanishing gradient. Another device, called ''patching``, appears in the papers
by Barron and Jensen \cite{bj} and by Crandall, Gunnarsson and Wang \cite{cgw}, and
it is based on the use of the eikonal equation. We show below that these two methods
actually coincide. This fact was mentioned in \cite{cgw}, but no proof for it was
given. The result will be crucial in the proof of our main uniqueness result,
Theorem \ref{thm:unique1} below.

To proceed,
we need some notation. We denote by
\[
\begin{split}
\Lip(u,B_r(x))=\inf\{L\in\R\,:\, \abs{u(z)-u(y)}\le L\abs{z-y}\trm{ for }z,y\in B_r(x)\}
\end{split}
\]
the least Lipschitz constant for $u$ on the ball $B_r(x)$. Let
$h\colon\Om\to\R$ be the unique viscosity solution to the infinity
Laplace equation $\Delta_\infty h=0$ in $\Om$ with boundary values
$h=f$ on $\partial\Om$. Then $h$ is everywhere differentiable, see
\cite{es}, and $\abs{Dh(x)}$ equals to the pointwise Lipschitz
constant of $h$,
\begin{equation*}
 L(h, x) := \lim_{r \to +0} \mathrm{Lip}(h, B_r(x))
\end{equation*}
for every $x\in\Om$. Since the map $x\mapsto L(h,x)$ is upper semicontinuous, see e.g.\ \cite{ceg},
this implies that the set
$$V_\eps := \{x \in \Om\colon \abs{Dh(x)} < \eps\}$$
is an open subset of $\Om$.
Now, define the ''patched solution``
$h_\eps\colon\overline\Om\to\R$ by first setting
$$
h_\eps=h\quad \trm{in }\overline\Om\setminus V_\eps,
$$
and then, for each connected component $U$ of $V_\eps$ and $x\in
U$, defining
$$
 h_\eps(x)
  = \sup_{y \in \partial U}
\left(
 h(y) - \varepsilon d_{U} (x, y)
\right),
$$
where $d_{U} (x, y)$ stands for the (interior) distance between
$x$ and $y$ in $U$. The results collected in the following
``patching lemma'' are taken from \cite{cgw}.

\begin{lemma}\label{lemma:patching} It holds that
\begin{enumerate}
 \item $\Delta_\infty h_\eps\ge 0$ in $\Om$ in the viscosity
 sense,
 \item $h_\eps = h$ on
          $\overline{\Omega} \setminus V_\varepsilon$ and
          $h_\eps \leq h$
          on $\overline{\Omega}$,
 \item  $L(h_\eps, x) \geq \eps$ for $x \in \Omega$,
\item $h_\eps$ is a viscosity solution to $\abs{Dh_\eps}-\eps=0$ in $V_\eps$.
\end{enumerate}
\end{lemma}

Now, let $z_\eps\colon\Om\to\R$ be the unique viscosity solution to Jensen's equation
\eqref{eq:jensen} in $\Om$ with $z_\eps =f$ on $\partial\Om$. Then it holds that the
patched solution coincides with the solution to
\eqref{eq:jensen}.
\begin{theorem}\label{thm:jensen_is_patch}
Let $z_\eps\in C(\overline\Om)$ be the solution to \eqref{eq:jensen} and $h_\eps$ be defined as above. Then
$z_\eps=h_\eps$ in $\Om$.
\end{theorem}

\begin{proof} Without loss of generality, we may assume that $\eps=1$.
Let us first show that $h_1$ is a supersolution to
\eqref{eq:jensen}. Let $\phi\in C^2(\Om)$ be such that $h_1-\phi$
has a local minimum at $x\in\Om$. If $x\in \Om\setminus V_1$, then
as $h(x)=h_1(x)$ and $h\ge h_1$ everywhere, we see that also
$h-\phi$ has a local minimum at $x\in\Om$. Since $h$ satisfies
$\Delta_\infty h=0$, this implies
\[
\min\{\Delta_\infty \phi(x), \abs{D\phi(x)}-1\}\le \Delta_\infty \phi(x) \le 0,
\]
as desired. On the other hand, if $x\in V_1$, then
$\abs{D\phi(x)}-1\le 0$ by Lemma~\ref{lemma:patching}, and again
we have
$$
\min\{\Delta_\infty \phi(x), \abs{D\phi(x)}-1\}\le 0.
$$
Thus $h_1$ is a supersolution to \eqref{eq:jensen}.

To prove that $h_1$ is a subsolution to \eqref{eq:jensen}, let
$\varphi\in C^2(\Om)$ be such that $h_1-\varphi$ has a local
maximum at $x\in\Om$. We need to prove that
$$\Delta_\infty \varphi(x)\ge
0 \trm{ and }\abs{D\varphi(x)}-1\ge 0.$$
 By Lemma \ref{lemma:patching}, the
first inequality holds no matter where $x$ lies, and the second is
true if $x\in V_1$. Thus we only have to show that
$$
\abs{D\varphi(x)}-1\ge 0
$$
if
$x\in \Om\setminus V_1$. But this is true because
\[
\max_{\overline B_r(x)}\varphi-\varphi(x) \ge \max_{\overline
B_r(x)}h_1-h_1(x) \ge r L(h_1,x) \ge r;
\]
see e.g.\ \cite{ceg} for the second inequality. The proof
of $z_1=h_1$ is completed upon recalling the uniqueness result for \eqref{eq:jensen} in \cite{J}.
\end{proof}

\section{Existence of solutions. A variational approach}
\label{sect-exist-variational}

In this section, we prove the existence of a viscosity solution to
the gradient constraint problem \eqref{bvp:grad_constraint} by
showing that solutions $u_p$ to
\begin{equation}\label{main.eq.p}
\begin{cases}
\Delta_p u= \chi_D \quad & \text{in $\Om$}\\
u=f \quad & \text{on $\partial\Om$}.
\end{cases}
\end{equation}
converge uniformly, as $p\to\infty$, to a function that satisfies
\eqref{bvp:grad_constraint}. In
fact, we prove a slightly more general result and show that the
convergence holds true even if we replace $\chi_D$ in
\eqref{main.eq.p} by a non-negative function $g\in L^\infty(\Om)$
for which $D=\{x\in\Om\colon g(x)>0\}$ and a certain non-degeneracy
assumption is valid, see \eqref{non-degene} below.

We begin by recalling the definitions of a weak and viscosity
solution for the equation $ \Delta_p u = g$, where $g\ge 0$ is
bounded but not necessarily continuous. Since we are mainly
interested in what happens when $p\to\infty$, we assume throughout
this section that $p> \max\{2,n\}$.

\begin{definition}\label{def:weaksol.p}
A function $u\in W^{1,p} (\Omega)\cap C(\Om)$ is a
\emph{weak solution} of  $\Delta_p u=g$ in $\Om$ if it satisfies
$$
- \int_\Omega |D u|^{p-2} D u \cdot D\varphi\,dx =
\int_\Omega g \varphi\,dx,
$$
for every $\varphi \in C^\infty_0 (\Omega)$.
\end{definition}

By a weak solution to the boundary value problem \eqref{main.eq.p} we mean a function
$u\in W^{1,p} (\Omega)\cap C(\overline\Om)$ that is a weak solution to  $\Delta_p u=g$ in $\Om$
and satisfies $u=f$ on $\partial\Om$. We also suppose that $f$ is
Lipschitz continuous to begin with, and fix a Lipschitz function
$F\colon\overline\Om\to\R$ such that $F=f$ on $\partial\Om$,
$\lip(F,\Om)=\lip(f,\partial\Om)$, and
$\norm{F}_{L^\infty(\Om)}=\norm{f}_{L^\infty(\partial\Om)}$. Since $p> n$ and $f$ is Lipschitz,
the conditions $u\in W^{1,p} (\Omega)\cap C(\overline\Om)$ and $u=f$ on $\partial\Om$
are equivalent to the statement that $u-F\in W^{1,p}_0(\Omega)$, see
\cite{m}.

\begin{lemma} \label{lemma.exis.p}
There exists a unique weak solution  $u\in W^{1,p} (\Omega)\cap C(\overline\Om)$ to  $\Delta_p u=g$ with
fixed Lipschitz continuous boundary values $f$, and it is
characterized as being a minimizer for the functional
$$
F_p (u) = \int_{\Omega} \frac{|D u|^p}{p} \, dx + \int_\Omega
g u \, dx
$$
in the set of functions $\{ u \in W^{1,p} (\Omega) \ : \ u =f \trm{ on } \partial \Om\}$.
\end{lemma}
\begin{proof} The functional $F_p$ is coercive and weakly semicontinuous, hence
the minimum is attained. Moreover, this minimum is
a weak solution to  $\Delta_p u=g$ in the sense of Definition
\ref{def:weaksol.p}. Uniqueness follows from the strict convexity of
the functional.
For details, we refer to Giusti's monograph \cite{Giusti}.
\end{proof}

Due to the possible discontinuity of the
right hand side $g(x)$, we use semicontinuous envelopes when defining viscosity solutions.
Denote
\[
g_*(x)=\liminf_{y\to x} g(y),
\]
and
\[
g^*(x)=\limsup_{y\to x} g(y).
\]
We recall that $g_*$ is lower semicontinuous, $g^*$ upper semicontinuous, and
$g_*\le g\le g^*$. We assume that $g$ is non-degenerate in the sense that
\begin{equation}\label{non-degene}
g_*(x)>0\qquad\text{for all $x\in\inter D$},
\end{equation}
with $D=\{x\in\Om\colon g(x)>0\}$. This condition is used in
Theorem~\ref{limit_is_visco}.

Notice that if $g(x)=\chi_D(x)$, then
$$
g_*(x)=\chi_{\inter D}(x) \quad \trm{ and }\quad g^*(x)=\chi_{\overline D}(x).
$$
Thus \eqref{non-degene} holds in this case.
Observe also that since we assume that $p\ge 2$, the
equation $\Delta_p u=g$ is not singular at the points where the
gradient vanishes, and thus $x\mapsto
\Delta_p \phi(x)=(p-2) |D \phi|^{p-4}\Delta_\infty \phi (x) + |D \phi|^{p-2}
\Delta \phi (x)$ is well
defined and continuous for any $\phi\in C^2(\Om)$.

\begin{definition}\label{def:viscosol.p}
An upper semicontinuous function $u\colon\Om\to\R$ is a
\emph{viscosity subsolution} to  $\Delta_p u=g$ in $\Om$ if, whenever
$ x \in\Om$ and $\varphi\in C^2(\Om)$ are such that $u-\varphi$
has a strict local maximum at $ x$, then
$$
\Delta_p \varphi ( x)
\ge g_*( x).
$$

A lower semicontinuous function $v\colon\Om\to\R$ is a
\emph{viscosity supersolution} to  $\Delta_p u=g$ in
$\Om$ if, whenever $ x\in\Om$ and $\phi\in C^2(\Om)$ are such that
$v-\phi$ has a strict local minimum at $ x$, then
$$
\Delta_p
\phi( x)\le g^*( x).
$$
Finally, a continuous function $h\colon\Om\to\R$ is a
\emph{viscosity solution} to  $\Delta_p u=g$ in $\Om$ if it is
both a viscosity subsolution and a viscosity supersolution.
\end{definition}

\begin{proposition} \label{weak.implies.viscosity}
A continuous weak solution of  $\Delta_p u=g$ is a viscosity
solution.
\end{proposition}

\begin{proof}
Let $x \in \Omega$ and choose a test function $\phi$ touching  $u$ at $x$
from below, that is, \ $u(x)=\phi (x)$ and
$u-\phi$ has a strict minimum at $x$.
We want to show that
$$
(p-2) |D \phi|^{p-4}\Delta_\infty \phi (x) + |D \phi|^{p-2}
\Delta \phi (x) \le g^*(x).
$$
If this is not the case, then there exists a radius $r>0$
such that
$$
(p-2) |D \phi|^{p-4}\Delta_\infty \phi (y) + |D \phi|^{p-2}
\Delta \phi (y) > g^*(y),
$$
for every $y\in B_r(x)$. Set $m = \inf_{|y- x|=r}
(u-\phi)(y)$ and let $\psi (y) = \phi(y) + m/2$. This function
$\psi$ verifies $\psi (x) > u(x)$ and
$$
\mbox{div} ( |D \psi|^{p-2} D \psi) > g^*(y)\ge g(y),
$$
which, upon integration by parts, implies that $\psi$ is a weak
subsolution to  $\Delta_p u=g$ in $B_r(x)$. Thus we have that
$\psi\le u$ on $\partial B_r(x)$, $u$ is a weak solution and
$\psi$ a weak subsolution to  $\Delta_p v=g$, which by the
comparison principle (for weak solutions) implies that $u\ge \psi$
in $B_r(x)$. But this contradicts the inequality $\psi (x) >
u(x)$.

This proves that $u$ is a viscosity supersolution. The proof of
the fact that $u$ is a viscosity subsolution runs along similar lines.
\end{proof}

Next we prove that there is a subsequence of weak solutions to
\begin{equation}
\label{eq:p-sol}
\begin{split}
\begin{cases}
\Delta_p u=g & \trm{in } \Om \\
u_p=f& \trm{on } \partial \Om
\end{cases}
\end{split}
\end{equation}
that converges uniformly as $p\to \infty$.

\begin{lemma}\label{uniform_convergence}
There exists a function $u_\infty
\in W^{1,\infty}(\Omega)$ and a subsequence of solutions to the above problem, \eqref{eq:p-sol}, such that
$$
\lim_{p \to \infty} u_{p} (x) = u_\infty (x)
$$
uniformly in $\overline{\Omega}$ as $p\to \infty$.
\end{lemma}

\begin{proof} Recall that we assumed
that $f$ is Lipschitz continuous. Let $h_p$ be the unique
$p$-harmonic function with boundary values $f$, that is, $h_p\in
W^{1,p}(\Om)$ satisfies
\begin{equation}\label{eq:pharm}
\begin{cases}
\Delta_p h_p = 0 & \text{in $\Om$}\\
h_p=f&\text{on $\partial\Om$}.
\end{cases}
\end{equation}
Now we can use $u_p-h_p\in W^{1,p}_0(\Om)$ (note that $u_p$ and
$h_p$ agree on $\partial \Omega$) as a test-function in the weak
formulations of
 $\Delta_p u=g$ and \eqref{eq:pharm}, and by subtracting the resulting equations obtain
\[
 \int_\Om (\abs{Du_p}^{p-2}Du_p-\abs{Dh_p}^{p-2}Dh_p)\cdot (Du_p-Dh_p)\, dx=\int_\Om g(u_p-h_p)\, dx.
\]
Using the well-known vector inequality
\[
2^{2-p}\abs{a-b}^p\le (\abs{a}^{p-2}a-\abs{b}^{p-2}b)\cdot(a-b),
\]
and H\"older's and Sobolev's inequalities (see
\cite[p.164]{gt} for the constants), this yields
\[
\begin{array}{l}
\displaystyle
\aveint{\Om}{} \abs{Du_p-Dh_p}^p\, dx \\[10pt]
\displaystyle \qquad \qquad
\leq 2^{p-2}\norm{g}_{L^\infty(\Omega)} \left(\frac{\abs{\Om}}{\omega_n}\right)^{1/n}
\left(\aveint{\Om}{} \abs{Du_p-Dh_p}^p\, dx\right)^{1/p},
\end{array}
\]
where $\omega_n$ is the measure of unit ball in $\R^n$,
$\abs{\Om}$ the measure of $\Om$, and $\aveint{}{}$ denotes the
averaged integral $\aveint{\Om}{}=\frac1{\abs{\Om}}\int_\Omega$.
Thus
\[
\left(\aveint{\Om}{} \abs{Du_p-Dh_p}^p\, dx\right)^{1/p} \le
 2^{\frac{p-2}{p-1}}\norm{g}_{L^\infty(\Om)}^{1/(p-1)}
\left(\frac{\abs{\Om}}{\omega_n}\right)^{1/n(p-1)},
\]
and if $n<m<p$, H\"older's inequality gives
\begin{equation}
\label{eq:W-infty-bound}
\begin{split}
\left(\aveint{\Om}{} \abs{Du_p-Dh_p}^m\, dx\right)^{1/m} \le 2
\max\left\{1,\norm{g}_{L^\infty(\Om)} \left(\frac{\abs{\Om}}{\omega_n}\right)^{1/n}\right\}.
\end{split}
\end{equation}
We infer from Morrey's inequality that the sequence
$\{u_p-h_p\}_{p\ge m}$ is bounded in $C^{0,1-n/m}(\overline\Om)$,
the space of $(1-n/m)$-H\"older continuous functions.
Thus, in view of Arzel\`a-Ascoli's theorem, there is $v\in C(\overline\Om)$ such that
(up to selecting a subsequence)
$u_p-h_p\to v$ as $p\to\infty$. Since the constant on the right-hand side of
\eqref{eq:W-infty-bound} is independent of $m$, a diagonal
argument shows that
$v\in W^{1,\infty}(\Om)$. Moreover, as
$h_p\to h$ uniformly in $\overline\Om$, where $h\in W^{1,\infty}(\Om)$ is the unique
solution to $\Delta_\infty h=0$ that agrees with $f$ on $\partial\Om$ (see \cite{J}),
we have that $u_p\to u_\infty:=v+h\in W^{1,\infty}(\Om)$ uniformly
in $\overline\Om$.
\end{proof}

\begin{theorem}\label{limit_is_visco}
A uniform limit $u_\infty$ of a subsequence $u_p$ as $p \to \infty$ is a
viscosity solution to \eqref{bvp:grad_constraint}.
\end{theorem}

\begin{proof}
From the uniform convergence it is clear that $u_\infty$ is
continuous and satisfies $u_\infty = f$ on $\partial \Omega$.

Next, to show that $u_\infty$ is a supersolution, assume that $u_{\infty} -\phi$ has a strict
minimum at $x\in \Omega$. We have to check that
\begin{equation}\label{superineqs.bb}
 \Delta_\infty \phi( x)\le 0\quad\text{or}\quad  \abs{D\phi( x)}
 -\chi_{\overline D}( x)\le 0.
\end{equation}

By the uniform convergence of $u_{p}$ to $u_{\infty}$ there are
points $x_{p}$ such that $u_{p} -\phi$ has a minimum at
$x_{p}$ and $x_{p} \to  x $ as $p \to \infty$. At those
points we have
$$
(p-2) |D \phi|^{p-4} \Delta_\infty \phi (x_{p}) + |D
\phi|^{p-2} \Delta \phi (x_{p}) \le g^*(x_{p} ).
$$

Let us suppose that $\abs{D\phi( x)}>\chi_{\overline D}(x)$, since
otherwise \eqref{superineqs.bb} clearly holds. Then $\abs{D\phi(
x)}>0$, and hence $\abs{D\phi(x_{p})}>0$ for $p$ large enough by
continuity. Thus we may divide by $(p-2) |D \phi(x_p)|^{p-4}$ in
the inequality above and obtain
\begin{equation}\label{pppp}
\Delta_\infty \phi (x_{p}) \le  \displaystyle\frac{1}{p-2} |D
\phi|^{2}
\Delta \phi (x_{p}) + \frac{g^* (x_{p}) }{(p-2) |D \phi(x_p)|^{p-4}}.
\end{equation}
Notice that if $ x\notin\overline{D}=\supp g$, then \eqref{pppp} implies
$\Delta_\infty\phi( x)\le 0$. On the other hand, if $x\in\overline{D}$, then
$\abs{D\phi( x)}> \chi_{\overline D}( x)=1$, which implies
$|D \phi(x_{p})|^{p-4}\to \infty$ as $p\to\infty$. In view of \eqref{pppp},
this gives again $\Delta_\infty\phi( x)\le 0$. Thus \eqref{superineqs.bb}
is valid.

To show that $u_\infty$ is also a subsolution to \eqref{bvp:grad_constraint},
we fix $\varphi\in C^2(\Om)$ such that $u_{\infty} -\varphi$ has a
strict maximum at $ x\in \Om$. We have to check that
\begin{equation}\label{subineqs.bb}
 \Delta_\infty \varphi ( x)\ge 0\quad\text{and}\quad  \abs{D\varphi( x)}
 -\chi_{\inter D}( x)\ge 0.
\end{equation}
By the uniform convergence of $u_{p}$ to $u_{\infty}$, there are
points $x_{p}$ such that $u_{p} -\varphi$ has a maximum at
$x_{p}$ and $x_{p} \to  x $ as $p \to \infty$. At those
points we have
\begin{equation}\label{p-ineq}
(p-2) |D \varphi|^{p-4} \Delta_\infty \varphi (x_{p}) + |D
\varphi|^{p-2} \Delta \varphi (x_{p}) \ge g_*(x_{p}).
\end{equation}

If $ x\notin\inter D$ and $\abs{D\varphi ( x)}=0$, then
\eqref{subineqs.bb} clearly holds. On the other hand, if $
x\in\inter D$, then by \eqref{non-degene}, $g_*(x_{p})>0$ for
$p$ large. In view of \eqref{p-ineq}, this implies that we must
have  $D\varphi (x_{p})\ne 0$. Thus, we can divide in
\eqref{p-ineq} by $(p-2) |D \varphi(x_p)|^{p-4}$ to obtain
\begin{equation}\label{pppp_and_one_more_p}
\Delta_\infty \varphi (x_{p}) \ge  -\displaystyle\frac{1}{p-2} |D
\psi|^{2}
\Delta \varphi (x_{p}) + \frac{g_*(x_{p}) }{(p-2) |D \varphi(x_p)|^{p-4}}.
\end{equation}
Since $g_*$ is non-negative, letting $p\to\infty$ in
\eqref{pppp_and_one_more_p} yields $\Delta_\infty \varphi ( x)\ge
0$. Moreover, if $ x\in\inter D$ and $\abs{D\varphi(
x)}<\chi_{\inter D}( x)=1$, then as $g_*( x)>0$, the right side of
\eqref{pppp_and_one_more_p} tends to infinity, whereas the left
side remains bounded, a contradiction. Therefore we must also have
 $\abs{D\varphi( x)} -\chi_{\inter D}( x)\ge 0$, which concludes the proof.
\end{proof}


Theorem \ref{limit_is_visco} says that the boundary value problem
\eqref{bvp:grad_constraint} has a solution if $f$ is assumed to be Lipschitz. Owing to
Jensen's uniqueness results, this restriction can be removed.

\begin{theorem}\label{thm:gen_existence}
The gradient constraint problem \eqref{bvp:grad_constraint} has at least one
solution for any $f\in C(\partial\Om)$ and $D\subset\Om$.
\end{theorem}

\begin{proof} Let $f_j$ be a sequence of Lipschitz functions converging
to $f$ uniformly on $\partial\Om$ and let $u_j$ be a solution to
\eqref{eq:grad_constraint} such that $u_j=f_j$ on $\partial\Om$, provided by
Theorem \ref{limit_is_visco}. Since $u_j$ is a subsolution to the
infinity Laplace equation, for every $\Om'\subset\subset\Om$ there
is $C>0$ such that $$
\norm{Du_j}_{L^\infty (\Om')}\le C,$$
see e.g.\ \cite{ACJ}. Thus, up to selecting a subsequence, there is a
locally Lipschitz continuous $u$ such that $u_j\to u$ locally uniformly.
Moreover, it follows from the standard stability results for viscosity
solutions, see \cite{CIL}, that $u$ is a solution to \eqref{eq:grad_constraint}.

Thus we only need to prove that $u=f$ on $\partial\Om$. Let $h_j$ and $z_j$
be the unique solutions to the infinity Laplace equation and Jensen's equation
\eqref{eq:jensen} (with $\eps=1$), respectively, such that $h_j=z_j=f_j$
on $\partial\Om$. By comparison, $z_j\le u_j\le h_j$ on $\overline\Omega$, and hence
$z\le u\le h$ on $\overline\Omega$, where $h$ and $z$
are the unique solutions to the infinity Laplace equation and
\eqref{eq:jensen}, respectively, with $h=z=f$ on $\partial\Om$ (cf. \cite[Corollary 3.14]{J}). In particular,
$u=f$ on $\partial\Om$, as desired.
\end{proof}

We close this section by proving a sharp a priori bound for the
solutions of \eqref{bvp:grad_constraint} that are obtained using
the $p$-Laplace approximation. To this end, we will again assume
that $f$ is Lipschitz and recall that $F$ denotes a Lipschitz
extension of $f$ that satisfies $\norm{F}_{L^\infty
(\Om)}=\norm{f}_{L^\infty (\partial\Om)}$.

\begin{lemma}\label{lem:lip_bound}
A uniform limit $u_\infty$ of a subsequence $u_p$ as $p \to \infty$ satisfies
\[
\norm{Du_\infty}_{L^\infty(\Om)} \le \max\{1,\lip(f)\}.
\]
\end{lemma}

\begin{proof}
Recall that we have denoted by $u_p$ the minimizer of
$$
F_p(u)=\frac1p\int_\Omega|D u|^p\, dx +\int_\Omega gu \, dx
$$
on the set $ K=\{u\in W^{1,p}(\Omega): \, u=f \text{ on
}\partial\Omega\}$. Then $F\in K$ and we have
\[
\begin{split}
\frac1p\int_\Omega|D u_p|^p\, dx +\int_\Omega gu_p \, dx \leq&\,
\frac1p\int_\Omega|DF|^p\, dx+\int_\Omega gF\, dx \\ \leq&\, \frac{\lip(f)^p
|\Omega|}{p} +
\norm{g}_{L^\infty(\Omega)}\norm{f}_{L^\infty(\partial\Omega)}\abs{\Om}.
\end{split}
\]
This together with H\"older's inequality implies
\begin{equation}
\label{eq:grad-bound}
\begin{split}
\int_\Omega|D u_p|^p\, dx \leq&\, \lip(f)^p
|\Omega| + pC - p\int_\Omega gu_p\, dx \\ \leq&\, \lip(f)^p
|\Omega| + pC + p \norm{u_p}_{L^p(\Om)} \norm{g}_{L^{p'}(\Om)}.
\end{split}
\end{equation}
By Sobolev's inequality, we have
\begin{equation}
\label{eq:Lp-bound}
\begin{split}
\norm{u_p}_{L^p(\Om)} \le&\, \norm{u_p-F}_{L^p(\Om)}+\norm{F}_{L^p(\Om)}\\
\le&\, C(\Om,n)\norm{Du_p-DF}_{L^p(\Om)}+\norm{f}_{L^\infty(\partial\Om)}\abs{\Om}^{1/p}\\
\le&\, C(\Om,n)\norm{Du_p}_{L^p(\Om)}+C(n,\Om)(\lip(f)+\norm{f}_{L^\infty (\partial \Omega)}).
\end{split}
\end{equation}
By combining \eqref{eq:grad-bound} and \eqref{eq:Lp-bound},
we obtain
$$
\int_\Omega|D u_p|^p \, dx
\leq  \lip(f)^p |\Omega| + pC + C p\|D u_p\|_{L^p(\Omega)},
$$
that is,
\begin{equation}
\label{eq:another-gradient-bound}
\begin{split}
\|D u_p\|_{L^p (\Omega)}
\leq  \left(C\lip(f)^p +  C p + C p\|D u_p\|_{L^p (\Omega)}\right)^{1/p},
\end{split}
\end{equation}
where the positive constant $C$ depends on $n$, $\Om$, $f$ and
$g$, but is independent of $p$ for $p> n$. Observe that
\[
(C a^p+p\, b+p\, c)^{1/p}\to \max\{a,1\}\quad\text{as $p\to\infty$}
\]
and recall the preliminary bound
\[
\begin{split}
\|D u_p\|_{L^p (\Omega)} \le&\, \|D h_p\|_{L^p (\Omega)}+ 2\abs{\Om}^{1/p}
\max\left\{1,\norm{g}_{L^\infty (\Omega)} \left(\frac{\abs{\Om}}{\omega_n}\right)^{1/n}\right\}\\
\le&\,C(n,\Om)(\lip(f)+\norm{g}_{L^\infty (\Omega)}+1)
\end{split}
\]
that was obtained in course of the proof of Lemma
\ref{uniform_convergence}. Combining these facts with \eqref{eq:another-gradient-bound}
we get
\[
\norm{Du_\infty}_{L^\infty (\Om)} \le \max\{1,\lip(f)\}.
\]
This ends the proof.
\end{proof}

\begin{remark} For future reference, we note that all the results in this section, except for
Theorem \ref{limit_is_visco}, hold for any bounded $g$ without any
sign restrictions.
\end{remark}


\section{Uniqueness and comparison results}\label{sec:uniqueness-and-comparison}
In this section, we show that under a suitable topological
assumption on $D$, the problem
\eqref{bvp:grad_constraint}, that is,  $\min\{\Delta_\infty u,
\abs{Du}-\chi_D\}=0$ with fixed boundary values, $u=f$ on $\partial \Omega$, has a unique
solution. In addition, if the condition is not satisfied, the
uniqueness is lost.

\begin{theorem}\label{thm:unique1}
Suppose that $\overline{\inter D}=\overline D$.
Then the problem \eqref{bvp:grad_constraint} has a unique solution.
\end{theorem}

To prove Theorem \ref{thm:unique1}, we show that a solution $u$ of
\eqref{bvp:grad_constraint} can be characterized
in the following way. Let $h\in C(\overline\Om)$ be the unique solution to
$\Delta_\infty h=0$ satisfying $h=f$ on $\partial\Om$,
and denote
\[
\mathcal{A}=\{x\in\Om\colon \abs{Dh(x)}<1\},\qquad \mathcal{B}=\mathcal{A}\cap D;
\]
recall that $h$ is everywhere differentiable, as proved in \cite{es}.
Let further $z\in C(\overline\Om)$ be the unique solution to the
Jensen's equation
\begin{equation}\label{eq:jensen1}
 \min\{\Delta_\infty z,\abs{Dz}-1\}=0,
\end{equation}
also satisfying $z=f$ on
$\partial\Om$. Then, if $\overline{\inter D}=\overline D$, we have
$u(x)=z(x)$ for all $x\in\mathcal{B}$ and
$\Delta_\infty u(x)=0$ in $\Om\setminus\overline{\mathcal{B}}$. But
the solution to $\Delta_\infty v=0$ in $\Om\setminus \ol {\mathcal B}$ with
the boundary values $z$ on $\partial \mathcal B$ and $f$ on
$\partial \Om$ is unique, and thus $u$ is unique.

\begin{theorem}\label{thm:charac} Suppose that $\overline{\inter D}=\overline D$.
Let $u\in C(\overline\Om)$ be a solution to \eqref{bvp:grad_constraint}.
Then $u(x)=z(x)$ for all $x\in\mathcal{B}$
and $\Delta_\infty u(x)=0$ in $\Om\setminus\overline{\mathcal{B}}$.
\end{theorem}

\begin{proof}
Observe first that since $u$ is a supersolution of the Jensen's
equation \eqref{eq:jensen1} and $u=z$ on $\partial\Om$, we have
$u\ge z$ in $\Om$. On the other hand, since $\Delta_\infty u\ge 0$
and $u=h$ on $\partial\Om$, we have $u\le h$ in $\Om$. Thus
\[
z(x)\le u(x) \le h(x)\quad\text{for all $x\in\Om$}.
\]
Next we recall Theorem \ref{thm:jensen_is_patch}, which implies
that $z(x)=h(x)$ in $\Om\setminus\mathcal{A}$.
This implies that
$$
u(x)=h(x)=z(x)\qquad \trm{in} \ \Om\setminus\mathcal{A},
$$
and,
in particular, that $\Delta_\infty u(x)=0$ in $\Om\setminus\overline{\mathcal{A}}$.
Moreover, as $\Delta_\infty u=0$ in $\Om\setminus\overline D$
by the fact that it satisfies \eqref{bvp:grad_constraint}, we have
$$\Delta_\infty u(x)=0\qquad \trm{in}\ \Om\setminus\overline{\mathcal{B}}.$$

To prove that $u=z$ in $\mathcal{B}$ we argue by contradiction and
suppose that there is $\hat x\in\inter\mathcal{B}$ such that
$u(\hat x)-z(\hat x)>0$. If $u$ were smooth, we would have $\abs{Du(\hat
x)}\ge 1$ by the second part of the equation, and from
$\Delta_\infty u\ge 0$ it would follow that $t\mapsto \abs{Du(\gamma(t))}$
is non-decreasing along the curve $\gamma$ for which $\gamma(0)=\hat x$
and $\dot{\gamma}(t)=Du(\gamma(t))$. Using this information and the fact that $\abs{z(x)-z(y)}\le
\abs{x-y}$ in $\mathcal A$, we could then follow $\gamma$ to $\partial\mathcal{A}$
to find a point $y$ where $u(y)>z(y)$; but this is a contradiction since
$u$ and $z$ coincide on $\partial \mathcal A$.

To overcome the fact that $u$ need not be smooth and to make the formal steps outlined above
rigorous, let $\delta>0$ and
\[
u_\delta(x)=\sup_{y\in\Om} \left\{
u(y)-\frac1{2\delta}\abs{x-y}^2\right\}
\]
be the standard sup-convolution of $u$. Observe that since $u$ is bounded in $\Om$, we in fact have
\[
u_\delta(x)=\sup_{y\in B_{R(\delta)}(x)} \left\{
u(y)-\frac1{2\delta}\abs{x-y}^2\right\}
\]
with $R(\delta)=2\sqrt{\delta\norm{u}_{L^\infty(\Omega)}}$. We
assume that $\delta>0$ is so small that
\begin{enumerate}
\item $\hat x\in (\inter\mathcal{B})_\delta := \{x\in
\inter\mathcal{B}\colon \dist(x,\partial\mathcal{B})>R(\delta)\}$; and
\item for $\mathcal{A}_\delta:=\{x\in \mathcal{A}\colon \dist(x,\partial\mathcal{A})>R(\delta)\}$, it holds
\[\sup\limits_{x\in (\inter\mathcal{B})_\delta} (u_\delta-z) >
\sup\limits_{x\in \partial\mathcal{A}_\delta} (u_\delta-z).\]
\end{enumerate}
Regarding the second condition, recall that $u_\delta\to u$
locally uniformly when $\delta\to 0$, and that $u=z$ on
$\partial\mathcal{A}$.

Next we observe that since $u$ is a solution to
\eqref{bvp:grad_constraint}, it follows that $\Delta_\infty
u_\delta\ge 0$ and $|Du_\delta|-\chi_{(\inter D)_\delta}\ge 0$ in
$\Om_\delta$;
 see e.g.\
\cite{jls}. In particular, since
$u_\delta$ is semiconvex and thus twice differentiable a.e., there
exists $x_0\in (\inter\mathcal{B})_\delta$ such that
$$u_\delta(x_0)-z(x_0) > \sup\limits_{x\in
\partial\mathcal{A}_\delta} (u_\delta-z),$$
$u_\delta$ is (twice)
differentiable at $x_0$, and
$$|Du_\delta(x_0)|=L(u_\delta,x_0)\ge
1.
$$
 Now let $r_0=\frac12\dist(x_0,\partial\mathcal{A}_\delta)$ and
let $x_1\in \partial B_{r_0}(x_0)$ be a point such that
\[
\max_{y\in \overline B_{r_0}(x_0)} u_\delta(y)=u_\delta(x_1).
\]
Since $\Delta_\infty u_\delta\ge 0$, the increasing slope estimate, see \cite{ceg}, implies
\[
1\le L(u_\delta,x_0) \le L(u_\delta,x_1)\quad \text{and}\quad u_\delta(x_1)
\ge u_\delta(x_0)+\abs{x_0-x_1}.
\]
By defining $r_1=\frac12\dist(x_1,\partial\mathcal{A}_\delta)$,
choosing $x_2\in \partial B_{r_1}(x_1)$ so that
$$\max\limits_{y\in \overline B_{r_1}(x_1)}
u_\delta(y)=u_\delta(x_2),$$ and using the increasing slope
estimate again yields
\[
1\le L(u_\delta,x_0) \le L(u_\delta,x_1)\le L(u_\delta,x_2)
\]
and
\[
u_\delta(x_2) \ge  u_\delta(x_1)+\abs{x_1-x_2} \ge u_\delta(x_0)+\abs{x_0-x_1}+\abs{x_1-x_2}.
\]
Repeating this construction gives a sequence $(x_k)$ such that
$x_k\to a\in \partial\mathcal{A}_\delta$
as $k\to\infty$ and
\[
u_\delta(x_k)\ge u_\delta(x_0)+\sum_{j=0}^{k-1} \abs{x_j-x_{j+1}} \quad\text{for $k=1,2,\ldots$}
\]
On the other hand, since $\abs{z(x)-z(y)}\le \abs{x-y}$ whenever the line segment
$[x,y]$ is contained in $\mathcal{A}$ (see \cite{cgw}),
we have
\[
z(x_k) \le z(x_0)+\sum_{j=0}^{k-1} \abs{x_j-x_{j+1}}.
\]
Thus, by continuity,
\[
u_\delta(a)-z(a) =\lim_{k\to\infty} u_\delta(x_k)-z(x_k) \ge u_\delta(x_0)-z(x_0) >
\sup\limits_{x\in \partial\mathcal{A}_\delta} (u_\delta-z).
\]
But this is impossible because $a\in \partial\mathcal{A}_\delta$. Hence the theorem is proved.
\end{proof}

\begin{remark}\label{rem:better_uniqueness}
The proof of Theorem \ref{thm:charac} shows that the uniqueness
for \eqref{bvp:grad_constraint} in fact holds under the weaker (but less explicit)
assumption
\[
\overline{\inter {\mathcal B}}=\overline {\mathcal B}.
\]
\end{remark}

\begin{remark}\label{rem:coincidence}
Under the assumption $\overline{\inter D}=\overline D$, we have that the unique solution
$u\in C(\overline\Om)$ to \eqref{bvp:grad_constraint} satisfies $u(x)=z(x)$ for all $x\in\overline D$.
This follows from the fact that for $x\in \Om\setminus\mathcal{A}$, $z(x)=h(x)$
by Theorem \ref{thm:jensen_is_patch}.
\end{remark}

In addition to uniqueness, we also have a comparison principle for
the equation $\min\{\Delta_\infty u, \abs{Du}-\chi_D\}=0$.

\begin{theorem}\label{thm:comparison1}
Suppose that $\overline{\inter D}=\overline D$. Then if $v_1$ is a subsolution and $v_2$ is a supersolution to
\eqref{bvp:grad_constraint}, we have $v_1\le v_2$.
\end{theorem}

\begin{proof}
Let $u$ be the unique solution to  \eqref{bvp:grad_constraint}. We will show that $v_1\le u$ and $u\le v_2$
in $\Om$.

Since $v_2$ is a supersolution to Jensen's equation
\eqref{eq:jensen1}, we have by Jensen's comparison theorem
\cite{J} that $v_2\ge z$ in $\Om$. In particular, owing to
Remark~\ref{rem:coincidence}, $v_2\ge u=z$ in $\overline D$. On
the other hand, in $\Om\setminus\overline D$ we have
$\Delta_\infty u=0$ by Theorem \ref{thm:charac} and $\Delta_\infty
v_2\le 0$, and so $v_2\ge u$ in $\Om\setminus\overline D$ as well.

As for the other inequality $u\ge v_1$, we notice first that it
suffices to prove that $u=z\ge v_1$ in $\mathcal{B}$. Indeed,
since $\Delta_\infty u=0$ in $\Om\setminus\mathcal B$ again by
Theorem~\ref{thm:charac} and  $\Delta_\infty v_1\ge 0$ in $\Om$,
it follows that
 $u\ge v_1$ in $\mathcal{B}$ implies  $u\ge v_1$ in $\Om$.

To prove $u=z\ge v_1$ in $\mathcal{B}$, we simply observe that we can repeat
the argument used in the second
part of the proof of Theorem \ref{thm:charac}, which proved that $u\le z$ in $\mathcal{B}$, as this
argument only used the fact that $u$ was a subsolution to \eqref{bvp:grad_constraint}.
\end{proof}




\subsection{Non-uniqueness}

In this section, we discuss various situations where there are more than one solution to
\eqref{bvp:grad_constraint}. For convenience of exposition, we assume that $\lip(f)\le 1$,
which implies that the infinity harmonic extension $h$ of $f$
satisfies $\norm{Dh}_{L^\infty(\Om)}\le 1$. This simplifies the
notation as we have
\[
D=D\cap\{\abs{Dh}\le 1\}.
\]

Let us first establish the non-uniqueness in the easy case $\inter D=\emptyset$.




\begin{lemma}\label{lem:empty_interior_gen_bdr_values}
Suppose that $\inter D=\emptyset$, and $f$ is Lipschitz continuous
with constant $L\le 1$. Let $v_\alpha\in C(\overline\Om)$ be the
unique function satisfying
\begin{equation}
\label{eq:least-solution}
 \begin{cases}
  \Delta_\infty v_\alpha=0&\text{in $\Om\setminus\overline D$}\\
v_\alpha = f&\text{on $\partial\Om$}\\
v_\alpha(x)=\sup\limits_{y\in \partial \Om}\Big(f(y)-\alpha
\abs{x-y}\Big)&\text{for $x\in\overline D$}.
 \end{cases}
\end{equation}
Then $v_\alpha$ is a solution to \eqref{bvp:grad_constraint} for
every $\alpha\in[L,1]$.
\end{lemma}

\begin{proof}
 We show
first that $v_\alpha$ is a subsolution to
 $\min\{\Delta_\infty u, \abs{Du}-\chi_D\}=0$. To this end, as $\inter D=\emptyset$,
we only have to verify that
$$\Delta_\infty v_\alpha\ge 0 \qquad \trm{in}\
\Om.$$ This is clearly true in $\Om\setminus \overline D$, so let
us suppose that $\hat x \in\Om$ and $\varphi\in C^2(\Om)$ are such
that $v_\alpha-\varphi$ has a strict local maximum at $\hat x\in
\overline D$. Since for each $y\in\partial\Om$,
\[
x\mapsto f(y)-\alpha \abs{x-y}
\]
is a viscosity subsolution of the infinity Laplace equation,
it follows that
$$g(x)= \sup_{y\in \partial
\Om}\Big(f(y)-\alpha \abs{x-y}\Big)$$ is a subsolution as well.
By the comparison principle, this implies that $v_\alpha(x)\ge
g(x)$ for all $x\in\Om$. In particular, since $v_\alpha(\hat x)=
g(\hat x)$, this means that also $g-\varphi$ has a strict local
maximum at $\hat x$. As $g$ is a subsolution, this implies
$\Delta_\infty
\varphi(\hat x)\ge 0$, as desired.

To show that $v_\alpha$ is a supersolution it suffices to prove
that $\abs{D v_\alpha}-1\le 0$ in $D$. But this is evident from
the fact that $\lip(v_\alpha,\overline\Om)=\alpha\le 1$,
which holds because $v_\alpha$ is the absolutely minimizing
Lipschitz extension of its boundary values to $\Om\setminus
\overline D$, see e.g.\ \cite{ceg}, and $L\le \alpha$.
\end{proof}

\begin{remark} In the special case of zero boundary values $f\equiv 0$,
Lemma \ref{lem:empty_interior_gen_bdr_values}
says that $v_\alpha$, the unique function that satisfies
\begin{equation}
\label{eq:least-solution.22}
 \begin{cases}
  \Delta_\infty v_\alpha=0&\text{in $\Om\setminus\overline D$}\\
v_\alpha(x)=0&\text{on $\partial\Om$}\\
v_\alpha(x)=-\alpha\dist(x,\partial\Om)&\text{for $x\in\overline D$},
 \end{cases}
\end{equation}
is a solution to \eqref{bvp:grad_constraint} for every $\alpha\in[0,1]$.

Note also that when $\inter (D) = \emptyset$ and $\overline{D}$
contains more than a single point, there are more solutions  to
\eqref{bvp:grad_constraint} than just the ones described in Lemma
\ref{lem:empty_interior_gen_bdr_values}. In fact, let us take any
subset $A
\subset \overline{D}$ and let $z$ be the solution to
\begin{equation}
 \begin{cases}
  \Delta_\infty z=0&\text{in $\Om\setminus A$}\\
z=0&\text{on $\partial\Om$}\\
z=-\dist(x,\partial\Om)&\text{for $x\in A$}.
 \end{cases}
\end{equation}
Then $\alpha z$ is also a solution to \eqref{bvp:grad_constraint},
with $f=0$, for every $\alpha\in[0,1]$.
\end{remark}

Next we will show that the uniqueness question
in the general case reduces to the uniqueness for
\eqref{bvp:no_interior} (see below), where $D\setminus \ol{\inter D}$ has empty interior.

If  $\inter D\ne\emptyset$, then   $\overline{\inter D}$ satisfies
the condition of Theorem \ref{thm:unique1} for uniqueness. That
is, $\overline{\inter (\overline{\inter
D})}=\overline{\overline{\inter D}}$, and thus there exists a
unique solution $u_0$ to
\begin{equation}\label{bvp_prime}
\begin{cases}
 \min\{\Delta_\infty u , \abs{Du}-\chi_{\overline{\inter D}}\}=0&\text{in $\Om$}\\
u=f&\text{on $\partial\Om$}.
\end{cases}
\end{equation}
Define $f_0\in C(\partial(\Om\setminus \overline{\inter D}))$ by setting
$f_0(x)=f(x)$, if $x\in\partial\Om$, and $f_0(x)=u_0(x)$, if $x\in\partial (\overline{\inter D})$.

\begin{lemma}\label{lem:reduction}
Suppose that $\overline{\inter D}\ne \overline D$ and $\inter D\ne\emptyset$.
Let $v$ be any solution to
\begin{equation}\label{bvp:no_interior}
\begin{cases}
 \min\{\Delta_\infty v, \abs{Dv}-\chi_{D\setminus \overline{\inter D}}\}=0&
 \text{in $\Om\setminus \overline{\inter D}$}\\
v=f_0&\text{on $\partial (\Om\setminus \overline{\inter D})$},
\end{cases}
\end{equation}
and define $w\colon\Om\to\R$ by
\[
w(x)=
\begin{cases}
 v(x),&\quad\text{if $x\in\Om\setminus \overline{\inter D}$}\\
u_0(x),&\quad\text{if $x\in \overline{\inter D}$}.
\end{cases}
\]
Then $w$ is a solution to
\[
\begin{cases}
 \min\{\Delta_\infty w, \abs{Dw}-\chi_{D}\}=0&\text{in $\Om$}\\
w=f&\text{on $\partial\Om$}.
\end{cases}
\]
\end{lemma}

\begin{proof}
Let us first show that $w$ is a supersolution. Let $\phi\in
C^2(\Om)$ be a function touching $w$ from below at $x\in\Om$.
If $x\in{\Om\setminus\ol{\inter D}}$, then, as $w=v$ in ${\Om\setminus\ol{\inter D}}$, we have
\[
\min\{\Delta_\infty \phi(x), \abs{D\phi(x)}-\chi_{\overline {D\setminus \ol{\inter D}}}(x)\}\le 0.
\]
But then also
\[
\min\{\Delta_\infty \phi(x), \abs{D\phi(x)}-\chi_{\overline {D}}(x)\}\le 0,
\]
as desired. On the other hand, if $x\notin{\Om\setminus\ol{\inter D}}$, then $x\in {\ol{\inter D}}$,
and we have $w(x)=u_0(x)$. Moreover, as $u_0$ is a solution and $v$ a subsolution to $\Delta_\infty h=0$ in
$\Om\setminus \ol{\inter D}$, we have $v\le u_0$ in $\Om\setminus \ol{\inter D}$ by comparison principle.
Thus $w\le u_0$ in $\Om$ and $\phi$ touches $u_0$ from below at
$x$, and thus
\[
\min\{\Delta_\infty \phi(x), \abs{D\phi(x)}-\chi_{\overline {{\inter D}}}(x)\}\le 0.
\]
This clearly implies
\[
\min\{\Delta_\infty \phi(x), \abs{D\phi(x)}-\chi_{\overline {D}}(x)\}\le 0,
\]
and we have shown that $w$ is a supersolution.

To check that $w$ is also a subsolution, we fix a function
$\varphi\in C^2(\Om)$ touching $w$ from above at $x\in\Om$.
We want to show that
\[
\Delta_\infty \varphi(x)\ge 0\quad\text{and}\quad
\abs{D\varphi(x)}\ge \chi_{\inter {D}}(x).
\]
First, if $x\in\inter D=\inter(\ol{\inter D})$, then $w=u_0$ in a neighborhood of $x$, and thus
\[
\min\{\Delta_\infty \varphi(x), \abs{D\varphi(x)}-\chi_{\inter ({{\overline{\inter D}}})}(x)\}\ge 0,
\]
by \eqref{bvp_prime}.
On the other hand, if $x\not\in\inter D$, then it suffices
to show that $\Delta_\infty \varphi(x)\ge 0$. This is
clearly true if $x\in{\Om\setminus\ol{\inter D}}$, so we may assume that $x\in\partial ({\overline{\inter D}})$.
But by Theorem \ref{thm:charac},
$u_0(x)=z(x)$, where $z$ is the unique solution to Jensen's equation \eqref{eq:jensen1},
and $u_0\ge z$ in $\Om$, so that $\varphi$ touches
also $z$ from above at $x$. Hence $\Delta_\infty
\varphi(x)\ge 0$, and we are done.
\end{proof}

Lemma \ref{lem:reduction} above shows that the uniqueness question
in the general case reduces to the uniqueness for
\eqref{bvp:no_interior}. This type of situation was already dealt
with in Lemma \ref{lem:empty_interior_gen_bdr_values}. However,
the reader should notice that in \eqref{bvp:no_interior},
$\lip(f_0)=1$, and thus Lemma
\ref{lem:empty_interior_gen_bdr_values} cannot be used to deduce
that there are more than one solution.

Next we present examples showing that under conditions $\overline{\inter D}\ne \overline D$
and $\inter D\ne\emptyset$ (and $f\equiv 0$), problem \eqref{bvp:grad_constraint} may have
either a unique solution or multiple solutions, depending on the geometry.

\begin{example} {\rm Suppose that $\Om=B_2(0)$, $f\equiv 0$, and $D=B_1(0)\cup D_1$, where
$D_1\subset B_2\setminus \overline B_1$ is any non-empty set with empty interior.
Clearly $\overline{\inter D}\ne \overline D$ and $\inter D\ne\emptyset$. We claim that the only
solution to 
$
\min\{\Delta_\infty z, \abs{Dz}-\chi_{D}\}=0
$ is $u(x)=\abs{x}-2$.
First, as $\abs{x}-2$ is a solution to  
$$\min\{\Delta_\infty u, \abs{Du}-1\}=0,$$
it follows that any solution
$v$ to 
$$\min\{\Delta_\infty v, \abs{Dv}-\chi_{D}\}=0$$
satisfies $v\ge u$, since $v$ is a supersolution to the first equation as well.

 On the other hand, it follows from Theorem \ref{thm:charac} that $u$ is the unique solution to
\[
\min\{\Delta_\infty u, \abs{Du}-\chi_{B_1(0)}\}=0
\]
satisfying the boundary condition $u=0$ on $\partial\Om$.
Thus, owing to Theorem
\ref{thm:comparison1}, we also have $v\le u$, since $v$ is a subsolution to the previous equation as
well.}
\end{example}

For the second example, we need the following lemma:

\begin{lemma}\label{lem:max_solution}
Let $u_0$ be the unique solution to
\begin{equation}\label{bvp:no_interior2}
\begin{cases}
 \min\{\Delta_\infty u, \abs{Du}-\chi_{\overline{\inter D}}\}=0&\text{in $\Om$}\\
u=f&\text{on $\partial\Om$}.
\end{cases}
\end{equation}
Then $u_0$ is the largest solution to \eqref{bvp:grad_constraint}.
\end{lemma}

\begin{proof}
We only need to show that $u_0$ is a solution to
\eqref{bvp:grad_constraint} as its maximality then follows
directly from the comparison principle, Theorem \ref{thm:comparison1}. To this end, we
first show that $u_0$ is a supersolution. Let $\phi\in C^2(\Om)$
be a function touching $u_0$ from below at $ x\in\Om$. Then,
by
\eqref{bvp:no_interior2},
\[
0\ge \min\{\Delta_\infty \phi, \abs{D\phi}-\chi_{\ol{\inter D}}\}\ge \min\{\Delta_\infty \phi,
\abs{D\phi}-\chi_{\ol D}\}.
\]%
Thus $u_0$ is a supersolution to \eqref{bvp:grad_constraint}.

The subsolution case is quite similar.  Let $\varphi\in C^2(\Om)$ be a
function touching $u_0$ from above at $ x\in\Om$. By $\inter(\overline{\inter D})=\inter D$
and \eqref{bvp:no_interior2}, it follows that
\[
0\le\min\{\Delta_\infty \varphi, \abs{D\varphi}-\chi_{\inter(\overline{\inter D})}\}
=\min\{\Delta_\infty \varphi, \abs{D\varphi}-\chi_{\inter D}\}.
\]
Thus $u_0$ is also a subsolution to \eqref{bvp:grad_constraint}.
\end{proof}

\begin{example}\label{ex:nonuniqueness} {\rm
Let $n=2$, $\Om=\mathopen]-1,1\mathclose[^2$, $f\equiv 0$, and
$D=D_0\cup\{x_0\}$, where $D_0=B_{1/2}(0)$ and $x_0\in \Om\setminus \overline B_{1/2}(0)$ is to be chosen.

Let $u$ be the unique solution to
\[
\begin{cases}
 \min\{\Delta_\infty u, \abs{Du}-\chi_{D_0}\}=0&\text{in $\Om$}\\
u=0&\text{on $\partial\Om$}.
\end{cases}
\]
Then, by Lemma \ref{lem:max_solution}, $u$ is a solution to \eqref{bvp:grad_constraint}.
Moreover, by Theorem \ref{thm:charac},
$$
u(x)=-\dist(x,\partial\Om)
$$
for all $x\in \overline D_0$.
By Lemma \ref{lem:reduction}, equation \eqref{bvp:no_interior}, which in this case reads
\begin{equation}
\label{bvp:reduction_in_this_case}
\begin{split}
\begin{cases}
\min\{\Delta_\infty v, \abs{Dv}-\chi_{\{x_0\}}\}=0&\text{in }\Om\setminus \ol{\inter D}=\Om\setminus \ol D_0\\
v=u&\text{on }\partial(\Om\setminus  \ol D_0),
\end{cases}
\end{split}
\end{equation}
determines the uniqueness of a solution to \eqref{bvp:grad_constraint}.
We show that there exist several solutions for this problem.

Choose $x_0\in \Om\setminus \overline D_0$ so that $-\dist(x,\partial\Om)$
is not differentiable at $x_0$, and let $u_1$ be the unique function such that
\[
\begin{split}
\begin{cases}
u_1(x)=-\dist(x,\partial\Om),&x\in \overline D_0\cup\{x_0\}\\
u_1=0,&\trm{on } \partial\Om\\
\Delta_\infty u_1=0,& \trm{in } \Omega\setminus (\overline D_0\cup\{x_0\}).
\end{cases}
\end{split}
\]
We can now check that $u_1$ is a solution to
\eqref{bvp:reduction_in_this_case}.
Indeed, that $u_1$ is subsolution follows easily, because $\inter\{x_0\}=\emptyset$,
$\Delta_\infty u_1=0$ in $\Omega\setminus (\overline D_0\cup\{x_0\})$, and
there are no $C^2$ functions touching $u_1$ from above at $x_0$.


To show that $u_1$ is also a supersolution, let $\phi\in
C^2(\Om\setminus \ol D_0)$ be a function touching $u_1$ from below
at $ x\in\Om\setminus \ol D_0$. The case $x\in \Om\setminus (\ol
D_0\cup \{x_0\})$ is again clear, and we may assume that $ x=x_0$.
Since $\norm{Du_1}_{L^\infty(\Om)}\le 1$ by the definition of
$u_1$ and properties of infinity harmonic functions, it follows
that $\abs{D\phi(x_0)}\le 1=\chi_{\overline D}(x_0)$. Hence we
have shown that $u_1$ is also a supersolution to
\eqref{bvp:grad_constraint}.

Finally, let us observe that $u(x_0)\ne u_1(x_0)$. The reason is
that $u$ is, as a solution to the infinity Laplace equation $\Delta_\infty h=0$, differentiable at
$x_0$ (see \cite{es}), but $u_1$ is not, because it touches
$-\dist(x,\partial\Om)$ from above at $x_0$.}
\end{example}

\subsection{Minimal and maximal solutions}
It turns out that the boundary value problem \eqref{bvp:grad_constraint} has always
a maximal and a minimal solution. The maximal solution $\bar u$ has been characterized in Lemma
\ref{lem:max_solution}, and the minimal solution $\underline{u}$ can be constructed as follows.
Given any $D\subset\Om$, define
\[
D_i=(D+B(0,{\tfrac1i}))\cap\Om=\{x\in\Om\colon \dist(x,D)<\tfrac1i\}.
\]
Then $D_i$ is open and hence $\overline{D_i}=\overline{\inter D_i}$. By Theorem \ref{thm:unique1},
there is a unique solution $u_i$ to
\[
\begin{cases}
 \min\{\Delta_\infty u, \abs{Du}-\chi_{D_i}\}=0&\text{in $\Om$}\\
u=f&\text{on $\partial\Om$}.
\end{cases}
\]
By the comparison principle, Theorem \ref{thm:comparison1}, the sequence $(u_i)$ is
monotone, and, as each $u_i$ is a subsolution to the infinity Laplace equation,
locally equicontinuous. Moreover, $u_1\le u_i\le h$ in $\overline\Om$, where
$\Delta_\infty h=0$ in $\Om$ and $h=f$ on $\partial\Om$.
Hence $u_i\to u$ locally uniformly in $\Om$ for some $u$ satisfying $u=f$ on $\partial\Om$.
It follows easily from the standard stability results for viscosity solutions that $u$ is
a solution to \eqref{bvp:grad_constraint}. Moreover, Theorem \ref{thm:comparison1}
implies that any solution $v$ to \eqref{bvp:grad_constraint} satisfies $v\ge u_i$ for all $i$'s,
and hence $u$ must be the minimal solution. We will give an alternative characterization
for the minimal solution in Section \ref{sec:games} below.

In Theorem \ref{limit_is_visco}, we proved the existence of a solution
to \eqref{bvp:grad_constraint} using $p$-Laplace approximation. Next we show that
this ``variational'' solution is, in general, neither the minimal nor the maximal solution.

\begin{example} {\rm
Let $\Om=B_1(0)$, $D=\{0\}$, and $f\equiv 0$. Then, by Lemma \ref{lem:empty_interior_gen_bdr_values},
the maximal solution is $\bar u(x)\equiv 0$ and the minimal solution $\underline{u}(x)=\abs{x}-1$.
Since $\abs{D}=0$, it follows that the unique solution $u_p$ to
\[
 \begin{cases}
  \Delta_p v=\chi_D&\text{in $\Om$},\\
v=0&\text{on $\partial\Om$},
 \end{cases}
\]
is $u_p\equiv 0$. Hence, in this case, the (unique) variational solution is the maximal solution
 $\bar u(x)\equiv 0$.

Now if $\Om$ and $f$ are as above, and $D=\Om\cap
(\R\setminus\mathbb{Q})^n$ (that is, $D$ consists of points in
$\Om$ with irrational coordinates), then, again by Lemma
\ref{lem:empty_interior_gen_bdr_values}, the maximal solution is
$\bar u(x)\equiv 0$ and the minimal solution
$\underline{u}(x)=\abs{x}-1$. However, since $\chi_D(x)=1$ a.e.,
it follows that $u_p$ is a solution also to $\Delta_p v=1$ in
$\Om$. Then it is well-known, see e.g.\ \cite{BBM} and the references therein, that
$u_p\to -\dist(x,\partial\Om)=\underline{u}(x)$ as
$p\to\infty$. Hence, in this case, the (unique) variational
solution is the minimal solution $\underline{u}(x)$.}
\end{example}

\begin{example} {\rm
Let $\Om=B_4(0)$, $f\equiv 0$, and $D=\{0\}\cup \hat D$, where $\hat D=(B_3\setminus B_2)\cap(\R\setminus
\mathbb{Q})^n$ (that is, $\hat D$ consists of the points in the annulus $B_3\setminus B_2$ having
irrational coordinates). Then, as $\inter D=\emptyset$, $\bar
u\equiv 0$ and $\underline{u}(x)=4-\abs{x}$. However, the results
in \cite{BBM} imply that $$u_\infty (x)=\lim\limits_{p\to\infty}
u_p (x)= -\dist(x,\partial\Om)$$ in $B_3\setminus B_2$ and that
$u_\infty$ is a solution to the infinity Laplace equation in
$B_2$; thus in this case the (unique) variational solution
$u_\infty$ is neither the minimal nor the maximal solution. }
\end{example}

\section{Games}\label{sec:games}

In this section, we consider a variant of the tug-of-war game
introduced by Peres, Schramm, Sheffield and Wilson in \cite{PSSW},
and show that the value functions of this game converge, as the
step size tends to zero, to the minimal solution of
\eqref{bvp:grad_constraint}.

As before, let $\Om$ be a bounded open set and $D\subset \Omega$.
For a fixed $\eps>0$, consider the following two-player
zero-sum-game. If $x_0\in\Om\setminus D$, then the players play a
tug-of-war game as described in \cite{PSSW}, that is, a fair coin
is tossed and the winner of the toss is allowed to move the game
token to any $x_1\in\ol B_\eps (x_0)$. On the other hand, if
$x_0\in D\cap \Om$, then Player II, the player seeking to minimize
the final payoff, can either sell the turn to Player~I with the
price $-\eps$ or decide that they toss a fair coin and play
tug-of-war. If Player II sells the turn, then Player I can move
the game token to any $x_1\in\ol B_\eps (x_0)$. After the first
round, the game continues from $x_1$ according to the same rules.

This procedure yields a possibly infinite sequence of game states
$x_0,x_1,\ldots$ where  every $x_k$ is a random variable. The game
ends when the game token hits $\Gamma_\eps$, the boundary strip of
width $\eps$ given by
\[
\begin{split}\Gamma_\eps=
\{x\in \RR^n \setminus \Om\,:\,\dist(x,\partial \Om)< \eps\}.
\end{split}
\]
We denote by $x_\tau \in \Gamma_\eps$ the first point in the
sequence of game states that lies in $\Gamma_\eps$, so that $\tau$
refers to the first time we hit $\Gamma_{\eps}$.

At this time the game
ends with the terminal payoff given by $F(x_\tau)$, where
$F:\Gamma_\eps
\to
\R$ is a given
Borel measurable
continuous
\emph{payoff function}. Player I earns $F(x_\tau)$ while Player II
earns $-F(x_\tau)$.

A strategy $S_\I$ for Player I is a function defined on the
partial histories that gives the next game position $
S_\I{\left(x_0,x_1,\ldots,x_k\right)}=x_{k+1}\in \ol B_\eps(x_k)$
if Player I wins the toss. Similarly Player II plays according to
a
strategy $S_\II$. 
In addition, we define a decision variable,
which tells when Player II decides to sell a turn
\[
\begin{split}
\theta(x_0,\ldots,x_k)=\begin{cases} 1,&x_k\in D \trm{ and Player
II sells a turn,}\\
0,& \trm{otherwise}.
\end{cases}
\end{split}
\]

The one step transition probabilities will be
 \[
  \begin{split}
\pi_{S_\I,S_\II,\theta}&(x_0,\ldots,x_k,{A})\\
&= \big(1-\theta(x_0,\ldots,x_k)\big)\frac{1}{2}\Big(
\delta_{S_\I(x_0,\ldots,x_k)}({A})+\delta_{S_\II(x_0,\ldots,x_k)}({A})\Big)\\
&\hspace{1
em}+\theta(x_0,\ldots,x_k)\delta_{S_\I(x_0,\ldots,x_k)}(A).
\end{split}
\]
By using the Kolmogorov's extension theorem and the one step transition probabilities, we can build a
probability measure $\mathbb{P}^{x_0}_{S_\I,S_\II,\theta}$ on the
game sequences. The expected payoff, when starting from $x_0$ and
using the strategies $S_\I,S_\II$, is
\begin{equation}
\label{eq:defi-expectation}
\begin{split}
\mathbb{E}_{S_{\I},S_\II,\theta}^{x_0}&\left[F(x_\tau)-\eps\sum_{i=0}^{\tau-1}
\theta(x_0,\ldots, x_i)\right]\\
 &=\int_{H^\infty} \Big(F(x_\tau)-\eps
\sum_{i=0}^{\tau-1} \theta(x_0,\ldots, x_i)\Big) \ud
\mathbb{P}^{x_0}_{S_\I,S_\II,\theta},
\end{split}
\end{equation}
where $F\colon\Gamma_\eps\to\R$ is the given continuous function
prescribing the terminal payoff.

The \emph{value of the game for Player I} is given by
\[
u^\eps_\I(x_0)=\sup_{S_{\I}}\inf_{S_\II,\theta}\,
\mathbb{E}_{S_{\I},S_\II,\theta}^{x_0}\left[F(x_\tau)-\eps\sum_{i=0}^{\tau-1}
\theta(x_0,\ldots,x_i)\right]
\]
while the \emph{value of the game for Player II} is given by
\[
u^\eps_\II(x_0)=\inf_{S_\II,\theta}\sup_{S_{\I}}\,
\mathbb{E}_{S_{\I},S_\II,\theta}^{x_0}\left[F(x_\tau)-\eps\sum_{i=0}^{\tau-1}
\theta(x_0,\ldots, x_i)\right].
\]
Intuitively, the values $u^\eps_\I(x_0)$ and $u^\eps_\II(x_0)$ are the best
expected outcomes
 each player can  guarantee when the game starts at
$x_0$. Observe that if the game does not end almost surely, then
the expectation \eqref{eq:defi-expectation} is undefined. In this
case, we define $\mathbb{E}_{S_{\I},S_\II,\theta}^{x_0}$ to take
value $-\infty$ when evaluating $u^\eps_\I(x_0)$ and $+\infty$
when evaluating $u^\eps_\II(x_0)$.

We start the analysis of our game with the statement of the  {\it
Dynamic Programming Principle} (DPP).

\begin{lemma}[DPP]
\label{lem:DPP}
 The value function for Player I satisfies for
$x\in \Om$
\[
\begin{split}
u^\eps_\I(x) = \min \Big\{ \frac12  \sup_{y\in \ol B_\eps (x)}
u^\eps_\I(y) + \frac12 \inf_{y\in \ol B_\eps (x)} u^\eps_\I(y) ;
 \sup_{y\in \ol B_\eps (x)} u^\eps_\I(y) - \eps \chi_{D}(x)\Big\}
\end{split}
\]
and  $u_\I^\eps(x)=F(x)$ in $\Gamma_\eps$. The value function
for Player II, $u^\eps_\II$, satisfies the same equation.
\end{lemma}

If $u^\eps_\I= u^\eps_\II$, we say that the game has a
value. Our game has a value, but we postpone the proof of this
fact until Theorem~\ref{thm:D-game-value}. First we prove that the
value $u_\eps$ of the game converges to the minimal solution of
\eqref{bvp:grad_constraint}.

\begin{theorem} \label{thm:game-gives-minimal-solution}
Let $u_\eps$ be the family of game values for a Lipschitz continuous
boundary data $F$, and let $u$ be the minimal solution to
\eqref{bvp:grad_constraint} with $F=f$ on $\partial\Om$. Then
\[
\begin{split}
u_\eps \to u \quad \trm{uniformly in }\ol \Om
\end{split}
\]
as $\eps\to0$.
\end{theorem}
As a first step, we prove that, up to selecting a subsequence, $u_\eps\to u$
as $\eps\to 0$ for some Lipschitz function $u$.

\begin{theorem}\label{thm:D-game-convergence}
Let $u_\eps$ be a family of game values for a Lipschitz continuous
boundary data $F$. Then there exists a Lipschitz continuous
function $u$ such that, up to selecting a subsequence,
\[
\begin{split}
u_\eps \to u \quad \trm{uniformly in }\ol \Om
\end{split}
\]
as $\eps\to0$.
\end{theorem}
\begin{proof}
Since $\Om$ is bounded, it suffices to prove asymptotic
Lipschitz continuity for the family $u_\eps$ and then use the asymptotic
version of Arzel\`a-Ascoli lemma from
\cite{MPR1}. We prove the required oscillation estimate by arguing by contradiction:
If there exists a point where the oscillation
$$
A(x) := \sup_{y\in
\ol B_\eps (x)} u_\eps (y) - \inf_{y\in \ol B_\eps (x)}
u_\eps (y)
$$
is large compared to the oscillation of the boundary data, then
the DPP takes the same form as for the standard tug-of-war game. Intuitively,
the tug-of-war never reduces the oscillation when playing to
$\sup$ or $\inf$ directions. Thus we can iterate this idea up to
the boundary to show that the oscillation of the boundary data must
be larger than it actually is, which is the desired contradiction.

To be more precise, we claim that
$$
A(x) \leq 4 \max \{ \lip (F) ; 1 \} \eps,
$$
for all $x\in \Om$.
Aiming for a contradiction, suppose that there exists $x_0\in
\Omega$ such that
$$
A(x_0) > 4 \max \{ \lip (F) ; 1 \} \eps.
$$
In this case, we have that
\begin{equation}
\label{eq:tug-of-war-like-dpp}
\begin{split}
u_\eps(x_0)
&=\min \Big\{ \frac12
\sup_{
\ol B_\eps (x_0)}u_\eps (y) + \frac12 \inf_{ \ol B_\eps (x_0)}
u_\eps (y) ;
 \sup_{ \ol B_\eps (x_0)}u_\eps (y) - \eps \chi_{D}\Big\} \\
&= \frac12
\sup_{
\ol B_\eps (x_0)}u_\eps (y) + \frac12 \inf_{ \ol B_\eps (x_0)}
u_\eps (y).
\end{split}
\end{equation}
The reason is that the alternative
$$
 \frac12
\sup_{y\in
\ol B_\eps (x_0)}u_\eps (y) + \frac12 \inf_{y\in \ol B_\eps (x_0)}
u^\eps (y) >
 \sup_{y\in \ol B_\eps (x_0)}u_\eps (y) - \eps \chi_{D}
$$
would imply
\begin{equation}
\label{eq:contradiction}
\begin{split}
A(x_0) =\sup_{y\in \ol B_\eps (x_0)}u_\eps (y) - \inf_{y\in \ol
B_\eps (x_0)}u_\eps (y) < 2\eps \chi_{D} \leq 2\eps,
\end{split}
\end{equation}
which is a contradiction with $A(x_0) > 4 \max \{ \lip (F) ; 1 \} \eps$.
It follows from \eqref{eq:tug-of-war-like-dpp} that
$$
\sup_{y\in
\ol B_\eps (x_0)}u_\eps (y) -u_\eps (x_0) =u_\eps(x_0) -\inf_{y\in \ol B_\eps
(x_0)}u_\eps (y) =\frac12 A(x_0) .
$$
Let $\eta>0$ and take $x_1 \in \ol B_\eps (x_0)$ such
that
$$
u_\eps (x_1) \geq \sup_{y\in \ol B_\eps (x_0)}u_\eps (y) -
\frac{\eta}{2}.
$$
We obtain
$$
u_\eps (x_1) -u_\eps (x_0) \geq
\frac12 A(x_0)-
\frac{\eta}{2} \geq 2\max \{ \lip (F) ; 1 \} \eps - \frac{\eta}{2},
$$
and, since $x_0\in \ol B_\eps (x_1)$, also 
$$
\sup_{y\in \ol B_\eps (x_1)}u_\eps (y) - \inf_{y\in \ol B_\eps (x_1)}u_\eps (y)  \geq 2\max
\{ \lip (F) ; 1 \} \eps - \frac{\eta}{2}.
$$
Arguing as before, \eqref{eq:tug-of-war-like-dpp} also holds at $x_1$, since otherwise the above
inequality would lead to a contradiction similarly as \eqref{eq:contradiction} for small enough $\eta$.
Thus
$$
\sup_{y\in
\ol B_\eps (x_1)}u_\eps (y) -u_\eps (x_1)
= u_\eps (x_1)- \inf_{y\in \ol B_\eps (x_1)}u_\eps (y)\geq 2\max
\{ \lip (F) ; 1 \} \eps - \frac{\eta}{2},
$$
so that
\[
\begin{split}
A(x_1)& =
\sup_{y\in
\ol B_\eps (x_1)}u_\eps (y)-u_\eps (x_1) +u_\eps (x_1) - \inf_{y\in \ol B_\eps (x_1)}
u_\eps (y) \\
&\geq 4\max
\{ \lip (F) ; 1 \} \eps - \eta.
\end{split}
\]
Iterating this procedure, we obtain $x_i \in \ol B_\eps (x_{i-1})$
such that
\begin{equation}\label{big_slope}
u_\eps (x_i) -u_\eps (x_{i-1}) \geq 2\max \{ \lip (F) ; 1
\}
\eps -
\frac{\eta}{2^i}
\end{equation}
and
\begin{equation}\label{big_osc}
A(x_i) \geq 4\max
\{ \lip (F) ; 1 \} \eps - (\sum_{j=0}^{i-1}  \frac{\eta}{2^j}).
\end{equation}

We can proceed with an analogous argument considering points where
the infimum is nearly attained to obtain $x_{-1}$, $x_{-2}$,...
such that $x_{-i} \in \ol B_\eps (x_{-(i-1)})$, and
\eqref{big_slope} and \eqref{big_osc} hold.
Since $u_\eps$ is bounded, there must exist $k$ and $l$ such that
$x_k, x_{-l}\in\Gamma_\eps$, and we have
$$
\begin{array}{rl}
\displaystyle \frac{|F (x_k) - F(x_{-l}) | }{| x_k - x_{-l}|} & \displaystyle  \geq
\frac{\displaystyle \sum\limits_{j=-l+1}^k u_\eps (x_{j}) -u_\eps (x_{j-1}) }{\eps
(k+l)}\\
& \displaystyle
\geq 2\max \{ \lip (F) ; 1
\} -
\frac{2\eta}{\eps},
\end{array}
$$
a contradiction.
Therefore
$$
A(x) \leq 4 \max \{\lip (F) ; 1 \} \eps,
$$
for every $x\in \Omega$.
\end{proof}

In order to prove Theorem~\ref{thm:game-gives-minimal-solution},
we define a modified game: the difference to the previous game is that
Player II can sell turns in the whole $\Om$ and not just when the token is in $D$.
We refer to our original game as $D$-game and to the modified game as $\Om$-game.

\begin{lemma}[DPP, $\Om$-game]
\label{lem:DPP-Om-game}
 The value function for Player I satisfies for
$x\in \Om$
\[
\begin{split}
u^\eps_\I(x) = \min \Big\{ \frac12  \sup_{y\in \ol B_\eps (x)}
u^\eps_\I(y) + \frac12 \inf_{y\in \ol B_\eps (x)}u^\eps_\I(y) ;
 \sup_{y\in \ol B_\eps (x)} u^\eps_\I(y) - \eps \Big\}
\end{split}
\]
and  $u^\eps_\I(x)=F(x)$ in $\Gamma_\eps$. The value function
for Player II, $u^\eps_\II$, satisfies the same equation.
\end{lemma}
It will be shown in Theorem~\ref{thm:Om-game-value} that this game has a value
$u_\eps=u^\eps_\I=u^\eps_\II$. We start by showing that the value of the game converges to the
unique solution of Jensen's equation $\min\{\Delta_\infty u, \abs{Du}-1\}=0$.

\begin{theorem}
\label{thm:Om-game-to-jensen-sol}
Let $u_\eps$ be the family of values of the $\,\Om$-game for a
Lipschitz continuous boundary data $F$, and let $u$ be the unique solution to \eqref{eq:jensen1} with
$u=F$ on $\partial\Om$. Then
\[
\begin{split}
u_\eps \to u \quad \trm{uniformly in }\ol \Om.
\end{split}
\]
\end{theorem}
\begin{proof}
The convergence of a subsequence to a Lipschitz continuous function $u$ follows by the
same argument as in Theorem~\ref{thm:D-game-convergence}.

By pulling towards a boundary point, we see that $u=F$ on $\partial \Omega$, and we can
focus our attention on showing that $u$ satisfies Jensen's equation in
the viscosity sense. To establish this,  we consider an asymptotic expansion related to our operator.

Fix a point $x\in \Omega$ and $\phi\in C^2(\Om)$.
Let $x_1^\eps$ and $x_2^\eps$ be a minimum point and a maximum point, respectively,
for $\phi$ in $\ol B_\eps(x)$. As in \cite{MPR}, we obtain the asymptotic expansion
\begin{equation}
\label{eq:asymp-jensen}
\begin{split}
&\min\Bigg\{\frac{1}{2}\max_{y\in \ol B_\eps (x)}
\phi (y) + \frac{1}{2}\min_{y\in \ol B_\eps(x)} \phi
(y);\max_{y\in \ol B_\eps (x)}
\phi (y) -\eps\Bigg\}-
\phi(x)\\
&\ge \min\Bigg\{\frac{\eps^2}{2} D^2\phi(x)\left(\tfrac{x_1^{\eps}-x}{\eps}\right)\cdot
\left(\tfrac{x_1^{\eps}-x}{\eps}\right)+o(\eps^2);\\
& \qquad \qquad \left(D \phi(x)\cdot
\tfrac{x_2^\eps-x}{\eps}-1\right)\eps
+\frac{\eps^2}{2}
D^2\phi(x)\left(\tfrac{x_2^{\eps}-x}{\eps}\right)\cdot
\left(\tfrac{x_2^{\eps}-x}{\eps}\right)+o(\eps^2)\Bigg\}.
\end{split}
\end{equation}
Suppose that $u-\phi$ has a strict local minimum at $x$ and that $D \phi(x)\neq 0$.
By the uniform convergence, for any
$\eta_\eps>0$ there exists a sequence
$(x_{\eps})$ converging to $x$ such that
$$
u_{\eps} (x) - \phi (x) \geq u_{\eps} (x_{\eps}) - \phi
(x_{\eps}) -\eta_\eps,
$$
that is, $u_{\eps} - \phi $
has an approximate minimum at $x_{\eps}$.
Moreover, considering $\tilde{\phi}= \phi - u_{\eps}
(x_{\eps}) - \phi (x_{\eps})$, we may assume that $\phi
(x_{\eps}) = u_{\eps} (x_{\eps})$.
 Thus, by recalling the fact that $u_\eps$  satisfies the DPP, we obtain
\[
\begin{split}
 \eta_\eps \geq-\phi (x_{\eps})&+
\min\Bigg\{\frac{1}{2}\max_{y\in \ol B_\eps (x)}
\phi (y) + \frac{1}{2}\min_{y\in \ol B_\eps(x)} \phi
(y);\max_{y\in \ol B_\eps (x)}
\phi (y) -\eps\Bigg\}.
\end{split}
\]
By choosing $\eta_\eps = o(\eps^2)$, using
\eqref{eq:asymp-jensen} and dividing by $\eps^2$, we have
\[
\begin{split}
0&\ge \min\Bigg\{\frac{1}{2} D^2\phi(x)\left(\tfrac{x_1^{\eps}-x}{\eps}\right)\cdot
\left(\tfrac{x_1^{\eps}-x}{\eps}\right)+\frac{o(\eps^2)}{\eps^2};\\
&\left(D \phi(x)\cdot \tfrac{x_2^\eps-x}{\eps}-1\right)\frac1\eps
+\frac{1}{2} D^2\phi(x)\left(\tfrac{x_2^{\eps}-x}{\eps}\right)\cdot
\left(\tfrac{x_2^{\eps}-x}{\eps}\right)+\frac{o(\eps^2)}{\eps^2}\Bigg\}.
\end{split}
\]
Since $D\phi(x)\neq 0$, by letting $\eps \to 0$, we conclude
that
\[
\begin{split}
\Delta_\infty \phi(x)\le 0 \quad \trm{or} \quad \abs{D\phi(x)}-1\le
0.
\end{split}
\]
This shows that $u$ is a viscosity supersolution to Jensen's equation \eqref{eq:jensen1}, provided
that $D \phi(x)\neq 0$. On the other hand, if $D
\phi=0$, then $\Delta_\infty \phi(x)=0$ and the same conclusion follows.

To prove that $u$ is a viscosity subsolution, we
consider a function $\varphi$ that touches $u$ from above at $x\in\Om$
and observe that a reverse inequality to
\eqref{eq:asymp-jensen} holds at the point of touching.
Arguing as above, one can deduce that
\[
\begin{split}
\Delta_\infty \varphi(x)\ge 0 \quad \trm{and} \quad
\abs{D\varphi(x)}-1\ge 0.
\end{split}
\]
To finish the proof, we observe that since the viscosity solution of \eqref{eq:jensen1}
is unique, all the subsequential limits of $(u_\eps)$ are equal.
\end{proof}

Next we prove that the value $u_\eps$ of the $D$-game converges to
the minimal solution of $\min\{\Delta_\infty u,
\abs{Du}-\chi_D\}=0$.
\begin{proof}[Proof of Theorem~\ref{thm:game-gives-minimal-solution}]
Let $h_\eps$, $u_\eps$ and $z_\eps$ denote the values of the standard
tug-of-war, the $D$-game, and the $\Om$-game, respectively.
 Since Player II has more options in $D$-game and again more in $\Om$-game, we have
\begin{equation}
\label{eq:ordered}
\begin{split}
z_\eps(x)\le u_\eps(x) \le h_\eps(x)\quad\text{for all $x\in\Om$}.
\end{split}
\end{equation}

As in the proof of Theorem~\ref{thm:charac}, we denote
by $h$ the unique solution to the infinity Laplace equation,
\[
\mathcal{A}=\{x\in\Om\colon \abs{Dh(x)}<1\},\qquad \mathcal{B}=\mathcal{A}\cap D,
\]
and by $z$ the unique solution to the Jensen's equation \eqref{eq:jensen1}.

We claim that
\[
\begin{split}
u_\eps\to z\quad\trm{in}\quad\mathcal B.
\end{split}
\]
Striving for a contradiction, suppose that there is $x_0\in \mathcal B$ such that
\begin{equation}
\label{eq:counterassumption}
\begin{split}
u_\eps(x_0)-z(x_0)>C
\end{split}
\end{equation}
for all $\eps>0$.
We recall from Theorem~\ref{thm:Om-game-to-jensen-sol} and \cite{PSSW} that
\[
\begin{split}
z_\eps \to z \quad \trm{and}\quad h_\eps\to h \quad \trm{uniformly in }\ol\Om.
\end{split}
\]
Moreover, Theorem
\ref{thm:jensen_is_patch} implies that $z(x)=h(x)$ in
$\Om\setminus\mathcal{A}$ and this together with \eqref{eq:ordered} yields
\begin{equation}
\label{eq:outside-A-infty-harmonic}
\begin{split}
h=z\le u_\eps +o(1)\le z+o(1)=h+o(1)
\end{split}
\end{equation}
in $\Om\setminus\mathcal{A}$ with a uniform error term.

We will next show that
\begin{equation}
\label{eq:delta(x)-large}
\begin{split}
\delta(x_0):=\sup_{y\in \ol B(x_0)} u_\eps(y)-u_\eps(x_0)\ge \eps.
\end{split}
\end{equation}
Indeed, looking at the DPP in Lemma~\ref{lem:DPP}, we have two
alternatives. The first alternative is
\[
\begin{split}
u_\eps(x_0)=\half \Big\{\inf_{y\in\ol B(x_0)} u_\eps(y)+\sup_{y\in \ol B(x_0)}
u_\eps(y)\Big\}<\sup_{y\in \ol B(x_0)} u_\eps(y)-\eps \chi_D(x_0).
\end{split}
\]
Since $x_0\in D$, this implies that
\[
\begin{split}
2 \eps<\sup_{y\in\ol B(x_0)} u_\eps(y)-\inf_{y\in\ol B(x_0)} u_\eps(y)
\end{split}
\]
and
\begin{equation}
\label{eq:up=down}
\begin{split}
\sup_{y\in \ol B(x_0)} u_\eps(y)-u_\eps(x_0)=u_\eps(x_0)-\inf_{y\in\ol B(x_0)} u_\eps(y),
\end{split}
\end{equation}
from which we deduce $\delta(x_0)>\eps$. The second alternative is
\begin{equation}
\label{eq:second-alternative}
\begin{split}
u_\eps(x_0)=\sup_{y\in \ol B(x_0)} u_\eps(y)-\eps
\end{split}
\end{equation}
which implies $\delta(x_0)=\eps$, and the claim \eqref{eq:delta(x)-large} follows.

Let $\eta>0$ and choose a point $x_1$
so that
\[
\begin{split}
u_\eps(x_1) \ge \sup_{y\in \ol B(x_0)} u_\eps(y)-\eta 2^{-1}.
\end{split}
\]
It follows from \eqref{eq:delta(x)-large} that $ u_\eps(x_1)-\inf_{y\in\ol B(x_1)} u_\eps(y)\ge \eps-\eta
2^{-1}$. Moreover, in the case of the first alternative at $x_1$ it holds that $\delta(x_1)
\ge \eps-\eta2^{-1} $ by the equation similar to \eqref{eq:up=down}, and in the case of the second
alternative, the equation similar to \eqref{eq:second-alternative} implies that
$\delta(x_1)=\eps$.

We iterate the argument and obtain a sequence of points $(x_k)$ such that
\begin{equation}
\label{eq:speed-for-u-eps}
\begin{split}
u_\eps(x_k)\ge u_\eps(x_0)+k \eps -\eta\sum_{i=1}^\infty2^{-i}.
\end{split}
\end{equation}
The sequence exits $\mathcal A$ in a finite number of steps, i.e.,
there exists a first point $x_{k_0}$ in the sequence such that $x_{k_0}\in
\Om\setminus \mathcal{A}$. This follows from \eqref{eq:speed-for-u-eps}
and the boundedness of $u_\eps$. On the other hand, since
$\abs{z(x)-z(y)}\le \abs{x-y}$ whenever the line segment $[x,y]$
is contained in $\mathcal{A}$ (see \cite{cgw}), we have
\[
z(x_{k_0}) \le z(x_0)+k\eps+C \eps,
\]
where the term $C\eps$ is due to the last step being partly outside $\mathcal A$.
By this estimate, \eqref{eq:outside-A-infty-harmonic} and \eqref{eq:speed-for-u-eps}, we obtain
\[
o(1)\ge u_\eps(x_{k_0})-z(x_{k_0}) \ge u_\eps(x_0)-z(x_0)-\eta-C\eps.
\]
This gives a contradiction with \eqref{eq:counterassumption} provided
we choose $\eta$ and $\eps$ small enough.

We have
\[
\begin{split}
u_\eps\to z \quad\trm{in}\quad \mathcal A\cap D \quad \trm{and}\quad u_\eps\to h
\quad\trm{in}\quad \Om\setminus \mathcal A.
\end{split}
\]
But in $\mathcal A\setminus D$, the $D$-game is just a tug-of-war,
and by \cite{PSSW}, $u_\eps$ converges to the unique
solution to
\[
\begin{split}
\begin{cases}
\Delta_\infty u=0,&\trm{in }\mathcal A\setminus D \\
u=h,& \trm{on }\partial A\setminus D\\
u=z,&\trm{on }\partial D\cap \mathcal A.
\end{cases}
\end{split}
\]
This ends the proof.
\end{proof}

Now let us show that there is a value for the $\Om$-game.

\begin{theorem}
\label{thm:Om-game-value}
The $\Om$-game has a value, i.e.\ $u^\eps_\I=u^\eps_\II$.
\end{theorem}
\begin{proof}
 By definition, $u^\eps_\I\le u^\eps_\II$,
and thus it remains to prove the opposite inequality. Observe that by
pulling towards a boundary point, Player II can end the game almost
surely and thus $u^\eps_\II<\infty$.
 Let
\[
\begin{split}
\delta(x):=\sup_{y\in\ol B_\eps(x)}u^\eps_\II(y) -u^\eps_\II(x).
\end{split}
\]
Suppose that Player I uses a strategy $S^0_\I$,
in which she always chooses to step to a point that almost
maximizes $u^\eps_\II$, that is, to a point $x_k$ such that
\[
u^\eps_\II(x_{k})\ge \sup_{y\in \ol B_{\eps}(x_{k-1})} u^\eps_\II(y)-\eta 2^{-k},
\]
for a fixed $\eta>0$. We claim that $m_k=u^\eps_\II(x_k)-\eta 2^{-k}$ is a submartingale.
Indeed, it follows by the DPP that
\begin{equation}
\label{eq:up-more-than-down}
\begin{split}
u^\eps_\II(x_k)-\inf_{\ol B_\eps (x_k)} u^\eps_\II(y)
\le \sup_{\ol B_\eps(x_k)} u^\eps_\II(y)-u^\eps_\II(x_k)=\delta(x_k),
\end{split}
\end{equation}
and thus
\[
\begin{split}
\mathbb{E}_{S^0_\I,
S_\II,\theta}^{x_0}&[u^\eps_\II(x_k)-\eta 2^{-k}|\,x_0,\ldots,x_{k-1}]\ge u^\eps_\II(x_{k-1})-\eta 2^{-(k-1)}.
\end{split}
\]
From the submartingale property it follows that the limit $\lim_{k\to \infty} m_{\tau\wedge k}$
exists by the martingale convergence theorem.
Furthermore, at every point $x\in \Om$ either
\[
\begin{split}
u^\eps_\II(x)=\half\left\{\inf_{y\in \ol B_\eps(x)} u^\eps_\II(y)+\sup_{y\in \ol B_\eps(x)}
u^\eps_\II(y)\right\}< \sup_{y\in \ol B_\eps(x)} u^\eps_\II(y)-\eps
\end{split}
\]
implying
\begin{equation}
\label{eq:slope-lower-bound-1}
\begin{split}
\eps<\sup_{y\in \ol B_\eps(x)} u^\eps_\II(y)-u^\eps_\II(x),
\end{split}
\end{equation}
or
\[
\begin{split}
u^\eps_\II(x)=\sup_{y\in \ol B_\eps(x)} u^\eps_\II(y)-\eps.
\end{split}
\]
Hence
\begin{equation}
\label{eq:slope-lower-bound-2}
\begin{split}
\sup_{y\in\ol B_\eps(x)} u^\eps_\II(y)-u^\eps_\II(x)=\eps.
\end{split}
\end{equation}
Thus $\delta(x)\ge \eps$ always. On the other hand, there are arbitrary long
sequences of moves made by Player I. Indeed, if Player II
sells a turn, then Player I gets to move, and otherwise this is a
consequence of the zero-one law.  Since $m_k$ is a bounded
submartingale, these two facts imply that the game must end almost surely.

By a similar argument utilizing the fact  $\delta(x)\ge \eps$, we see that
$$u^\eps_\II(x_k)-\eps\sum_{i=0}^{k-1}\theta(x_0,\ldots,x_i)-\eta
2^{-k}$$
is a submartingale as well. It then follows from Fatou's
lemma and the optional stopping theorem that
\[
\begin{split}
u^\eps_\I(x_0)
&= \sup_{S_{\I}}\inf_{S_{\II},\theta}\,\mathbb{E}_{S_{\I},S_{\II},\theta}^{x_0}[F(x_\tau)
-\eps\sum_{i=0}^{\tau-1}\theta(x_0,\ldots,x_i)]\\
&\ge \inf_{S_\II,\theta} \mathbb{E}_{S^0_\I,S_\II,\theta}^{x_0}[F(x_\tau)
-\eps\sum_{i=0}^{\tau-1}\theta(x_0,\ldots,x_i)-\eta 2^{-\tau}]\\
&\ge  \inf_{S_{\II},\theta} \limsup_{k\to\infty}\mathbb{E}_{S^0_\I,
S_\II,\theta}^{x_0}[u^\eps_\II(x_{\tau\wedge k})-\eps\sum_{i=0}^{\tau\wedge k-1}
\theta(x_0,\ldots,x_i)-\eta 2^{-(\tau\wedge k)}]\\
&\ge\inf_{S_{\II},\theta}  \mathbb{E}_{S^0_\I, S_\II,\theta}[u^\eps_\II(x_0)-\eta]=u^\eps_\II(x_0)-\eta.
\end{split}
\]
This implies that $u^\eps_\I\ge u^\eps_\II$.
\end{proof}

Now, let us prove the analogous statement for the $D$-game.

\begin{theorem}
\label{thm:D-game-value}
The $D$-game has a value,  i.e.\ $u^\eps_\I=u^\eps_\II$.
\end{theorem}
\begin{proof}
The proof is quite similar to the proof of
Theorem~\ref{thm:Om-game-value}, but when tug-of-war is played
outside $D$ we have to make sure that
\[
\begin{split}
\delta(x)=\sup_{y\in\ol B_\eps(x)} u^\eps_\II(y) -u^\eps_\II(x)
\end{split}
\]
is large enough. This is done by using the backtracking strategy,
cf. Theorem 2.2 of \cite{PSSW}.

Fix $\eta>0$ and a starting point $x_0\in \Om$, and set $\delta_0
=\min\{\delta(x_0),\eps\}/2$. We suppose for now that $\delta_0>0$, and
define
\[
\begin{split}
X_0=\Big\{x\in \Om\,:\, \delta(x)> \delta_0\Big\}.
\end{split}
\]
Observe that $D\subset X_0$ by estimates similar to \eqref{eq:slope-lower-bound-1}
and \eqref{eq:slope-lower-bound-2}.

We consider a strategy $S_I^0$ for Player I that distinguishes between the cases $x_k\in X_0$
and $x_k\notin X_0$. First, if $x_k\in X_0$, then she always chooses to step to a point
$x_{k+1}$ satisfying
\[
u^\eps_\II(x_{k+1})\ge \sup_{y\in \ol B_{\eps}(x_{k})} u^\eps_\II(y)-\eta_{k+1} 2^{-(k+1)},
\]
where $\eta_{k+1}\in (0,\eta]$ is small enough to guarantee that $x_{k+1}\in X_0$.
Thus if $x_k\in X_0$ and Player I gets to choose the next position (by winning the coin toss or
 through the selling of the turn by the other player), for
 $$m_k=u^\eps_\II(x_k)-\eta 2^{-k}$$ it holds that
 \[
 \begin{split}
 m_{k+1}&\ge u^\eps_\II(x_k)+\delta(x_k)-\eta_{k+1} 2^{-(k+1)}-\eta 2^{-(k+1)}\\
 &\ge u_\II^\eps(x_k)+\delta(x_k)-\eta 2^{-k}\\
 &=m_k+\delta(x_k).
 \end{split}
 \]
On the other hand, if Player II wins the toss and moves from $x_k\in X_0$ to $x_{k+1}\in X_0$,  it
holds, in view of \eqref{eq:up-more-than-down}, that
\[
\begin{split}
m_{k+1}\ge u^\eps_\II(x_k)-\delta(x_k)-\eta 2^{-(k+1)}>m_k-\delta(x_k). 
\end{split}
\]

In the case $x_k\notin X_0$, we set
\[
\begin{split}
m_k=u^\eps_\II(y_k)-\delta_0 d_k-\eta 2^{-k},
\end{split}
\]
where $y_k$ denotes the last game position in $X_0$ up to time $k$, and $d_k$ is the distance, measured in
number of steps, from $x_k$ to $y_k$ along the graph spanned by the previous points
$y_k=x_{k-j},x_{k-j+1},\ldots,x_k$ that were used to get from $y_k$ to $x_k$. The strategy for Player I
in this case is to backtrack to $y_k$,
that is, if she wins the coin toss, she moves the token to one of the points
$x_{k-j},x_{k-j+1},\ldots,x_{k-1}$ closer to $y_k$ so that $d_{k+1}= d_k-1$. Thus if Player I wins and
$x_k\notin X_0$ (whether $x_{k+1}\in X_0$ or not),
\[
\begin{split}
m_{k+1}\ge \delta_0+m_k.
\end{split}
\]

\sloppy To prove the desired submartingale property for $m_k$, there are three more cases to be checked.
If Player II wins the toss and he moves 
to a point $x_{k+1}\notin X_0$ (whether  $x_k\in X_0$ or not), it holds that
\[
\begin{split}
m_{k+1}&= u_\II^\eps(y_k)-d_{k+1} \delta_0-\eta 2^{-(k+1)}\\
&\ge u_\II^\eps(y_k)-d_{k} \delta_0-\delta_0-\eta 2^{-k}\\
&=m_k-\delta_0 .
\end{split}
\]
If Player II wins the coin toss and moves from $x_k\notin X_0$ to $x_{k+1}\in X_0$, then
\[
\begin{split}
m_{k+1}=u^\eps_\II(x_{k+1})-\eta 2^{-(k+1)}\ge -\delta(x_k)+u^\eps_\II(x_k)-\eta 2^{-k}\ge -\delta_0+m_k
\end{split}
\]
where the first inequality is due to \eqref{eq:up-more-than-down}, and the second follows from the fact
$m_k=u^\eps_\II(y_k)-d_k\delta_0-\eta 2^{-k}\le u^\eps_\II(x_k)-\eta 2^{-k}$.

Taking into account all the different cases, we see that $m_k$
is a bounded (from above) submartingale, and since Player I can assure that
$m_{k+1}\ge m_k+\delta_0$ if she wins a coin toss, the game must
again terminate almost surely. We can now conclude the proof similarly as
in the case of Theorem~\ref{thm:Om-game-value}; recall that $\delta(x_k)\ge \eps$ whenever $x_k$ in $D$.

Finally, let us remove the assumption that $\delta(x_0)>0$.
If $\delta(x_0)=0$ for $x_0\in X$, then
Player I adopts a strategy of pulling towards a boundary point
until the game token reaches
a point $x_0'$ such that $\delta(x_0')>0$ or $x_0'$ is outside
$\Om$. It holds that $u^\eps_\II(x_0)= u^\eps_\II(x_0')$, because by \eqref{eq:up-more-than-down} 
it cannot happen that
$\delta(x)=\sup_{y\in \ol B_\eps(x)} u^\eps_\II(y)-u^\eps_\II(x)=0$ and
$u^\eps_\II(x)-\inf_{y\in \ol B_\eps(x)} u^\eps_\II(y)>0$ simultaneously.
Thus we can repeat the proof also in this case.
\end{proof}


\section{$L^\infty$-viscosity solutions}

In this section, we outline another approach to the problem \eqref{bvp:grad_constraint}
\[
\begin{split}
\begin{cases}
 \min\{\Delta_\infty u, \abs{Du}-\chi_D\}=0&\text{in $\Om$}\\
u=f&\text{on $\partial\Om$}.
\end{cases}
\end{split}
\]
The idea is to regard $\chi_D$ as a bounded, measurable function, defined
only up to a set of measure zero, and to accommodate the set of test-functions to
this interpretation. This point of view fits well with the approximation
of \eqref{bvp:grad_constraint} by the equations $\Delta_p u_p=\chi_D$, but it turns
out to be incompatible with the game approach at least in some cases.

\subsection{The approximating
$p$-Laplace equations}

We begin by recalling the definition of $L^\infty$-viscosity solutions for the approximating
$p$-Laplace equations.
For simplicity, we consider only the equation
\begin{equation}\label{eq:plap}
 \Delta_p u=\chi_D,
\end{equation}
 and leave the more general version $\Delta_p u=g$, with $g$
non-negative and bounded, to the reader. As before, the boundary
conditions are understood in the classical sense. For more on
$L^\infty$-viscosity solutions, see e.g.~\cite{ccks}.

\begin{definition}\label{def:linfty_viscosol.p}
A continuous function $u\colon\Om\to\R$ is an
\emph{$L^\infty$-viscosity subsolution} of \eqref{eq:plap} if, whenever
$ x \in\Om$ and $\varphi\in W^{2,\infty}_{loc}(\Om)$ are such that
$u-\varphi$ has a strict local maximum at $ x$, then
$$
\esslimsup_{y\to x} \left(\Delta_p \varphi
(y)-\chi_{D}(y)\right)\ge 0.
$$

A continuous function $v\colon\Om\to\R$ is an
\emph{$L^\infty$-viscosity supersolution} of \eqref{eq:plap}
if, whenever $ x\in\Om$ and
$\phi\in W^{2,\infty}_{loc}(\Om)$ are such that $v-\phi$ has
a strict local minimum at $ x$, then
$$
\essliminf_{y\to  x} \left(\Delta_p \phi(y)-\chi_{D}(y)\right)\le 0.
$$

Finally, a continuous function $h\colon\Om\to\R$ is an
\emph{$L^\infty$-viscosity solution} of \eqref{eq:plap}
if it is both a viscosity subsolution and a viscosity
supersolution.
\end{definition}

\begin{proposition} \label{weak.implies.linfty.viscosity}
A continuous weak solution of \eqref{eq:plap} is an
$L^\infty$-viscosity solution.
\end{proposition}

\begin{proof}The proof is almost the same as that of Proposition \ref{weak.implies.viscosity}.
The difference is that the counter proposition holds almost everywhere. Nonetheless,
the proof utilizes weak solutions and thus a set of measure zero makes no difference.
We leave the details to the reader.
\end{proof}

\subsection{The gradient constraint problem}
\begin{definition}\label{def:linfty_viscosol.infty}
A continuous function $u\colon\Om\to\R$ is an
\emph{$L^\infty$-viscosity subsolution} of \eqref{eq:grad_constraint} if,
 whenever $ x \in\Om$ and
$\varphi\in W^{2,\infty}_{loc}(\Om)$ are such that $u-\varphi$ has
a strict local maximum at $ x$, then
$$
\esslimsup_{y\to x} \Big(\min\{\Delta_\infty \varphi (y),\abs{D\varphi(y)}-\chi_{D}(y)\}\Big)\ge 0.
$$

A continuous function $v\colon\Om\to\R$ is an
\emph{$L^\infty$-viscosity supersolution} of \eqref{eq:grad_constraint}
if, whenever $ x\in\Om$ and
$\phi\in W^{2,\infty}_{loc}(\Om)$ are such that $v-\phi$ has
a strict local minimum at $ x$, then
$$
\essliminf_{y\to  x} \Big(\min\{\Delta_\infty \phi (y),\abs{D\phi(y)}-\chi_{D}(y)\}\Big)\le 0.
$$

Finally, a continuous function $h\colon\Om\to\R$ is an
\emph{$L^\infty$-viscosity solution} of \eqref{eq:grad_constraint}
if it is both a viscosity subsolution and a viscosity
supersolution.
\end{definition}

Our next result says that $L^\infty$-viscosity solutions are
viscosity solutions (in the sense of Definition
\ref{def:viscosol}). This holds for subsolutions and
supersolutions as well. However, the converse is not true, as explained
after Corollary~\ref{cor:Linfty-uniquness}.
\begin{lemma}\label{lem:linfty_is_visco}
An $L^\infty$-viscosity subsolution $u$ to
\eqref{bvp:grad_constraint} is a viscosity subsolution. Similarly,
an $L^\infty$-viscosity supersolution $v$ to
\eqref{bvp:grad_constraint} is a viscosity supersolution.
\end{lemma}

\begin{proof} Let us first prove the claim about subsolutions. Let $\varphi\in C^2(\Om)$ and
$ x\in\Om$ be such that $u-\varphi$ has a strict local maximum at
$ x$. We want to show that
\begin{equation}\label{what_to_prove}
\Delta_\infty \varphi( x)\ge 0\quad\text{and}\quad  \abs{D\varphi( x)}-\chi_{\inter D}( x)\ge 0.
\end{equation}

Since $u$ is an $L^\infty$-viscosity subsolution to
\eqref{bvp:grad_constraint}, we have
\begin{equation}\label{eq:esslimsup_ineq}
\esslimsup_{y\to x} \Big(\min\{\Delta_\infty \varphi (y),\abs{D\varphi(y)}-\chi_{D}(y)\}\Big)\ge 0.
\end{equation}
Observe that since the map $y\mapsto \Delta_\infty \varphi (y)$ is
continuous, we immediately have $\Delta_\infty \varphi( x)\ge 0$.
If $ x\in\inter D$, then also $y\mapsto \chi_D(y)$ is
continuous in a neighborhood of $ x$, and
\eqref{eq:esslimsup_ineq} implies \eqref{what_to_prove} as desired.
On the other hand, if $ x\not\in\inter D$, then $\abs{D\varphi( x)}-\chi_{\inter D}( x)\ge 0$
holds trivially, and we are done.

To prove the supersolution case, let $\phi\in
C^2(\Om)$ and $ x\in\Om$ be such that $u-\phi$ has a strict
local minimum at $ x$. We want to show that
\begin{equation}\label{what_to_prove2}
\Delta_\infty \phi( x)\le 0\quad\text{or}\quad  \abs{D\phi( x)}-\chi_{\overline D}( x)\le 0.
\end{equation}

Since $u$ is an $L^\infty$-viscosity supersolution to
\eqref{bvp:grad_constraint}, we have
\begin{equation}\label{eq:essliminf_ineq}
\essliminf_{y\to x} \Big(\min\{\Delta_\infty \phi (y),\abs{D\phi(y)}-\chi_{D}(y)\}\Big)\le 0.
\end{equation}
Let us suppose that $\Delta_\infty \phi( x)> 0$. Then $\abs{D\phi( x)}>0$, and we must
have $ x\in\overline D$, for otherwise we would contradict \eqref{eq:essliminf_ineq}. In fact,
\eqref{eq:essliminf_ineq} implies that
\[
\abs{D\phi( x)}\le \esslimsup_{y\to x} \chi_{D}(y) \le 1 = \chi_{\overline D}( x),
\]
which completes the proof.
\end{proof}

Lemma \ref{lem:linfty_is_visco} implies that
if \eqref{bvp:grad_constraint} has a unique viscosity solution, then it also has a unique
$L^\infty$-viscosity solution. In particular, we have

\begin{corollary}
\label{cor:Linfty-uniquness}
Suppose that $\overline D=\overline{\inter D}$. Then \eqref{bvp:grad_constraint} has a unique
$L^\infty$-viscosity solution.
\end{corollary}

\begin{proof}
By Theorem \ref{thm:unique1}, \eqref{bvp:grad_constraint} has a unique
viscosity solution. But by Lemma \ref{lem:linfty_is_visco},
any  $L^\infty$-viscosity solution is a viscosity solution, and thus there can be at most one.
The existence of an $L^\infty$-viscosity solution follows from $L^p$ approximation, see
Lemma \ref{lem:limit_is_linfty_visco} below.
\end{proof}

It is quite obvious that the uniqueness for $L^\infty$-viscosity
solutions holds in certain cases where there are several viscosity
solutions to the problem \eqref{bvp:grad_constraint}. For example,
if $\abs{D}=0$, then $u$ is an $L^\infty$-viscosity solution to
\eqref{bvp:grad_constraint} if and only if it is a solution to the infinity Laplace equation.
On the other hand,
Lemma \ref{lem:empty_interior_gen_bdr_values} shows that if
$\abs{D}=0$ and $\lip(f,\partial\Om)<1$, then there are multiple
viscosity solutions to \eqref{bvp:grad_constraint}. This also
shows that a viscosity solution need not be an $L^\infty$-viscosity solution.

More generally, by mimicking the proof of Theorem \ref{thm:charac}, one can prove
the following

\begin{theorem} Let $D\subset\Om$ be the set in \eqref{bvp:grad_constraint} and suppose that
there exists $D'\subset\Om$ for which
$\overline{D'}=\overline{\inter D'}$ and the symmetric difference
$$
D\bigtriangleup D' = (D\setminus D')\cup (D'\setminus D)
$$
has measure zero. Then the problem \eqref{bvp:grad_constraint} has a
unique $L^\infty$-viscosity solution.
\end{theorem}

Finally, we address the question of existence of $L^\infty$-viscosity solutions.
Recall from Lemma \ref{uniform_convergence} that if $f$ is Lipschitz, there
exists a subsequence of $(u_p)$, where $\Delta_p u_p=\chi_D$ in $\Om$
and $u=f$ on $\partial\Om$, and a function $u_\infty
\in W^{1,\infty}(\Omega)$ such that
$$
\lim_{p \to \infty} u_{p} (x) = u_\infty (x)
$$
uniformly in $\overline{\Omega}$. We already know  that $u_\infty$
is a viscosity solution to \eqref{bvp:grad_constraint}, and next
we show that it is also an $L^\infty$-viscosity solution to this
equation.

\begin{lemma}\label{lem:limit_is_linfty_visco}
A uniform limit $u_\infty$ of a subsequence $u_p$ as $p \to \infty$ is an
$L^\infty$-viscosity solution to \eqref{bvp:grad_constraint}.
\end{lemma}

\begin{proof} That $u=f$ on $\partial \Omega$ is immediate from
the uniform convergence.

Now, let us first check that $u_\infty$ is an $L^\infty$-viscosity
subsolution. To this end, let us fix $\varphi\in
W^{2,\infty}_{loc}(\Om)$ such that $u-\varphi$ has a strict local
maximum at some $x\in\Om$. By the uniform convergence of
a subsequence $u_{p}$ to $u_{\infty}$ there are points $x_{p}$ such that
$u_{p} -\varphi$ has a minimum at $x_{p}$ and $x_{p} \to x $ as $p \to \infty$. At those points we have
\[
\esslimsup_{y\to x_{p}} \Big(\Delta_p\varphi(y)-\chi_D(y)\Big)\ge 0.
\]

We show first that
\begin{equation}\label{sub_part_one}
 \esslimsup\limits_{y\to  x}
\Delta_\infty\varphi(y)\ge 0.
\end{equation}
We argue by contradiction, and suppose
that there is $r>0$ and $\eps>0$ such that
\[
\Delta_\infty \varphi \le -\eps<0\quad\text{a.e.\ in $B_r( x)$}.
\]
Observe that this implies $\abs{D\varphi}>0$ a.e.\ in $B_r( x)$.
Denoting
$$M_1=\| {D\varphi}\|_{L^{\infty}(B_{2r}(x))} \qquad
\mbox{ and } \qquad M_2=\| {D^2\varphi} \|_{L^{\infty}(B_{2r}( x))},$$ we
have
\[
\begin{split}
\Delta\varphi+(p-2)\abs{D\varphi}^{-2}\Delta_\infty\varphi
\le nM_2 -(p-2)\frac\eps{M_1^2}
\end{split}
\]
a.e.\ in $B_r( x)$. In particular, for $p$ large enough this expression is negative, and hence we have that
\[
\begin{split}
\Delta_p\varphi=&\abs{D\varphi}^{p-2}\left(\Delta\varphi+(p-2)\abs{D\varphi}^{-2}\Delta_\infty\varphi\right)\\
\le& \left(\frac{\eps}{M_2}\right)^{(p-2)/2}
(nM_2 -(p-2)\frac\eps{M_1^2}) <0
\end{split}
\]
a.e.\ in $B_r( x)$. This contradicts the fact that
\[
\esslimsup_{y\to x_{p}} \Delta_p\varphi(y)\ge \esslimsup_{y\to x_{p}}
\Big(\Delta_p\varphi(y)-\chi_D(y)\Big)\ge 0,
\]
and thus \eqref{sub_part_one} must hold.

Next we show that
\begin{equation}\label{sub_part_two}
\esslimsup\limits_{y\to  x}
(\abs{D\varphi(y)}-\chi_D(y))\ge 0.
\end{equation}
We again argue by contradiction, and suppose that there is $r>0$ and $\eps>0$ such
that
\[
\abs{D\varphi}-\chi_D \le -\eps<0\quad\text{a.e.\ in $B_r( x)$}.
\]
Thus $\abs{B_r( x)\setminus D}=0$ and $\abs{D\varphi}\le 1-\eps$ a.e.\ in $B_r( x)$.
This implies that
\[
\begin{split}
\Delta_p\varphi-\chi_D=&\abs{D\varphi}^{p-2}\left(\Delta\varphi+(p-2)
\abs{D\varphi}^{-2}\Delta_\infty\varphi\right)-1\\
\le& (1-\eps)^{p-2}(n+p-2)M_2 -1,
\end{split}
\]
a.e.\ in $B_r( x)$. The last expression on the right is negative if
$p$ is large enough, and
we arrive to a contradiction by arguing as above. Hence \eqref{sub_part_two} is valid,
and together with \eqref{sub_part_one} this implies that $u_\infty$ is an $L^\infty$-viscosity
subsolution to \eqref{bvp:grad_constraint}.

To prove that $u_\infty$ is also an $L^\infty$-viscosity supersolution, we fix
$\phi\in W^{2,\infty}_{loc}(\Om)$ such that $u-\phi$ has a strict local
minimum at some $ x\in\Om$. Again set $M_1=\| {D\phi}\|_{L^{\infty}(B_{2r}(x))} $ and
$M_2=\| {D^2\phi} \|_{L^{\infty}(B_{2r}( x))}$.
We have to show that
\[
\essliminf\limits_{y\to  x} \Big(\min\{\Delta_\infty\phi(y),\, \abs{D\phi(y)}-\chi_D(y)\}\Big)\le 0.
\]

Suppose this is not the case. Then there are $r,\eps>0$ such that
\[
\Delta_\infty \phi (y)\ge \eps\quad\text{and}\quad \abs{D\phi(y)}-\chi_D(y)\ge \eps
\]
a.e.\ in $B_r( x)$. Then
\[
\Delta\phi+(p-2)\abs{D\phi}^{-2}\Delta_\infty\phi
\ge -nM_2+(p-2)\frac\eps{M_1^2}>0\quad\text{a.e.\ in $B_r( x)$}
\]
for $p$ large enough, and hence for such $p$'s,
\[
\begin{array}{l}
\displaystyle \Delta_p\phi-\chi_D= \abs{D\phi}^{p-2}\left(\Delta\phi
+(p-2)\abs{D\phi}^{-2}\Delta_\infty\phi\right)-\chi_D\\
\displaystyle \ge \left(\chi_D(y)+\eps\right)^{p-2}((p-2)\frac\eps{M_1^2}-nM_2)-\chi_D\\
\displaystyle
\ge \min\{\eps^{p-2}((p-2)\frac\eps{M_1^2}-nM_2),
\left(1+\eps\right)^{p-2}((p-2)\frac\eps{M_1^2}-nM_2)-1\}>0
\end{array}
\]
a.e.\ in $B_r( x)$. Recalling that by the uniform convergence of $u_{p}$ to $u_{\infty}$ there are
points $x_{p}$ such that $u_{p} -\phi$ has a minimum at
$x_{p}$ with $x_{p} \to  x $ as $p \to \infty$, and that $u_{p}$'s are
$L^\infty$-viscosity supersolutions to \eqref{main.eq.p}, we have a contradiction.
\end{proof}

\begin{section}{An application: asymptotic behavior for $p$-Laplace problems}
\label{sect-application}

Given functions $g\in L^\infty(\Omega)$ and
$f\colon\partial\Omega\to\R$ that is Lipschitz continuous, we
consider, for every $p>2$, the solution $u_p$ to the elliptic
problem
\begin{equation}
\label{eq:main}
\left\{
\begin{array}{ll}
 \Delta_p u =g \qquad &\text{ in }\Omega \\
  u=f \qquad &  \text{
on }\partial\Omega.
\end{array} \right.
\end{equation}
Our aim is to apply the results of the preceding sections to study the
limit as $p\to\infty$ of the functions $u_p$. In particular, we
want to see how this limit depends on the data $f$ and $g$.

The case $f=0$ was already considered in \cite{IL} (see also \cite{BBM}, \cite{jan}), where the
authors prove that there is a uniform limit that depends on $g$.
In particular, it is proved there that when $g$ does not change
sign and $f=0$ then the limit is unique and depends only on the support of $g$.

\subsection{The case $\lip(f,\partial\Om)\le 1$}
The solution to \eqref{eq:main} for a given $p$ admits a
variational characterization, namely, it is the unique minimizer
of the functional
$$
J_p(u)=\frac1p\int_\Omega|D u|^p \, dx+\int_\Omega gu \, dx
$$
in the set $ K_p=\{u\in W^{1,p}(\Omega): \, u=f \text{ on
}\partial\Omega\}$.

We proved in Lemma \ref{lem:lip_bound} that any subsequential limit
$u_\infty$ of $u_p$'s satisfies
$$
\|D u_\infty\|_{L^\infty (\Om)} \leq \max \{ \lip(f), 1\} =1,
$$
and from this it follows that $u_\infty$ minimizes the functional
$$
J_\infty (u)=\int_{\Omega}gu
$$
in the set $ K_\infty =\{u\in W^{1,\infty}(\Omega): \, \|
Du\|_{L^\infty (\Omega) } \leq 1 \, \mbox{ and }\, u=f
\text{ on }\partial\Omega\}$.
Indeed, since for any $v\in K_\infty$,
$$
\int_{\Omega} g u_p\, dx\le \frac1p\int_\Omega|D u_p|^p \, dx+\int_\Omega gu_p \, dx\le
\frac{|\Omega|}p+\int_\Omega gv \, dx,
$$
the claim follows from the uniform convergence $u_p \to u_{\infty }$.

In certain cases, the problem of minimizing $J_\infty$ has clearly
a unique solution. For example, if $g > 0$ in $\Omega$, then the unique
minimizer is given by
$$
u(x)=\max_{y\in\partial\Omega}\Big\{ f(y)-|x-y|\Big\}.
$$

However, if $\lip(f,\partial\Om)> 1$, then it is not so easy to
identify $u_\infty$ as a minimizer of some variational problem, and we
have to do something else.

\subsection{The general case}
Let $g$ be continuous and non-negative.
Then the non-degeneracy condition \eqref{non-degene} clearly holds, and
thus Lemma \ref{limit_is_visco} implies
that any subsequential limit $u_\infty$ of $u_p$'s is a viscosity
solution to
\begin{equation}\label{EQ.LIMITE}
\min\{\Delta_\infty u,|Du| - \chi_{D}\}=0,
\end{equation}
where $D=\{x\in\Om\colon g(x)>0\}$.
Since, by the continuity of $g$, we have
$\overline{\inter D}=\overline D$, Theorem \ref{thm:unique1} says that \eqref{EQ.LIMITE} has
a unique solution. Therefore, recalling Remark \ref{rem:better_uniqueness},
we have proved the following result:

\begin{theorem}\label{theo.f.g} Let $g\geq 0$ be continuous. Then,
the limit of the solutions $u_p$ as $p\to \infty$ is characterized
by being the unique solution to \eqref{EQ.LIMITE} with boundary
datum $f$. In particular, the limit depends on $g$ only through
the set
$$
\supp (g) \cap \{x\in\Om\colon |Dh(x)| <1\},
$$
where $h$ stands for the unique solution to the infinity Laplace equation with $h=f$ on $\partial\Om$.
\end{theorem}
\end{section}

\end{document}